\documentclass[pra,preprint,amsmath,amssymb]{revtex4}
\usepackage{amsthm}
\usepackage{bm,xcolor}

\usepackage{physics}
\usepackage{enumitem}
\usepackage{soul}

\usepackage{graphicx}
\usepackage{epstopdf}

\renewcommand{\i}{\ensuremath{\text{\normalfont I}}}
\newcommand{\ii}{\ensuremath{\text{\normalfont I\!I}}}
\newcommand{\iii}{\ensuremath{\Gamma}}
\newcommand{\ici}{\ensuremath{\text{\normalfont I,I}}}
\newcommand{\icii}{\ensuremath{\text{\normalfont I,I\!I}}}
\newcommand{\iici}{\ensuremath{\text{\normalfont I\!I,I}}}
\newcommand{\iicii}{\ensuremath{\text{\normalfont I\!I,I\!I}}}

\newcommand{\iiciii}{\ensuremath{\text{\normalfont I\!I,}\Gamma}}
\newcommand{\iiicii}{\ensuremath{\Gamma,\text{\normalfont I\!I}}}
\newcommand{\wh}{\widehat}

\newtheorem{theorem}{Theorem}[section]

{\bf}{\it}
\newtheorem{remark}[theorem]{Remark}
\begin{document}
\preprint{}

\title{Absorbing boundary conditions for the time-dependent Schr\"odinger-type equations in $\mathbb R^3$.}
\author{Xiaojie Wu}
\email{xiaojiewu@berkeley.edu}
\affiliation{Department of Mathematics, University of California, Berkeley}
\author{Xiantao Li}
\email{xxl12@psu.edu}
\affiliation{Department of Mathematics, The Pennsylvania State University}

\date{\today}
\begin{abstract}
  Absorbing boundary conditions are presented for three-dimensional time-dependent Schr\"odinger-type of equations as a means to reduce the cost of the quantum-mechanical calculations. The boundary condition is first derived from a semi-discrete approximation of the Schr\"odinger equation with the advantage that the resulting formulas are automatically compatible with the finite-difference scheme and no further discretization is needed in space.  The absorbing boundary condition is expressed as a discrete Dirichlet-to-Neumann (DtN) map, which can be further approximated in time by using rational approximations of the Laplace transform to enable a more efficient implementation. This approach can be applied to domains with arbitrary geometry.  The stability of the zeroth order and first order absorbing boundary conditions is proved. We tested the boundary conditions on  benchmark problems. The effectiveness is further verified by a time-dependent Hartree-Fock model with Skyrme interactions. The accuracy in terms of energy and nucleon density is examined as well.
\end{abstract}
\maketitle

\section{Introduction}
Quantum-mechanical simulations, expressed in terms of the Schr\"odinger equation, provide a fundamental description of physical properties in chemistry and condensed-matter physics \cite{martin2004, helgaker2014}. On the other hand, due to the small length scale, an outstanding challenge is the size of the system that one can simulate, even with the rapid growth of computing power. 
Computer simulations often face the scenarios where the electrons are being emitted out of the computational domain, e.g., the photoionization process. Recently there has been a great deal of renewed interest in such issues in the quantum transport problem \cite{xie2014complex}, optical response of molecules \cite{yabana2006real}, open quantum systems \cite{Hellumns1994}, etc.

One approach to address the aforementioned  issue is by absorbing boundary conditions (ABCs) which reduce the problem to a much smaller computational domain and yield results as if the simulation is being performed in a much larger (or unbounded) domain. Rather than simply removing the exterior region, the ABCs provide an efficient approach to mimic the influence from the surrounding environment.  There are several different approaches to construct and implement ABCs, most of which involve the derivation and approximation of the Dirichlet-to-Neumann (DtN) map. They are also commonly referred to as non-reflecting, transparent,  or radiating boundary conditions. By replacing the surrounding region with the ABCs, the computational effort can be focused on simulating the region of interest. 

There seemed to be separate development of ABCs for quantum-mechanical simulations. For example, there is a great deal of progress in developing ABCs in physics and chemistry. These  ABCs can be roughly classified into four approaches: 1) the exterior complex scaling (ECS) \cite{Scrinzi2014,NICOLAIDES197811,SIMON1979211,mccurdy2004solving}, where the coordinate outside a fixed radius is scaled into the complex plane; 2) the mask function method \cite{marques2003,tafipolsky2006,Wang201107}, where the wave functions are gradually scaled to zero; 3) complex absorbing potentials (CAP) \cite{riss1996investigation,Riss_1993,MUGA2004357}, which is introduced outside the boundary by adding a complex potential to the Hamiltonian; and 4) coordinate scaling, which scales the spatial coordinate by a time-dependent factor \cite{Soloviev1985,sidky2000} to make wave functions compatible with the physical boundary conditions. 
These methods are not constructed to directly approximate the exact boundary condition. Rather, they are very much goal-oriented. Their direct aim is to absorb the electrons that move outside the computational domain. These methods in general are easy to implement, and are quite robust in practice. Meanwhile, the effectiveness depends heavily on the choice of the parameters, e.g., in the CAP methods, the size of the boundary region and the magnitude of the potential. 
In some cases, the efficiency and accuracy (reflections) are even debatable  for decades \cite{Riss_1998, He2007, Tao2009, Scrinzi2010, Ge1998}. The ABCs can also be introduced for the density-matrix in the Liouville von-Neumann equation \cite{li2019absorbing}.

 In the applied math community, much effort has been focused on one-dimensional problems, with a few works for multi-dimensional problems with flat boundary, for a comprehensive review, see \cite{antoine2008review, han2013}. There, a different path is followed:  The {\it exact} ABCs are first derived by solving an exterior problem, many of which are written as a Dirichlet-to-Neumann (DtN) or Neumann-to-Dirichlet (NtD) map for the continuous time-dependent Schr\"odinger equation (TDSE) \cite{han2005finite} (or temporally discrete model \cite{antonie2003}, spatially discrete model \cite{alonso2002,alonso2003}, and fully discrete model \cite{arnold1998}). The same technique has been used by Ermolaev et al. \cite{ermolaev1999, ermolaev2000} to treat electron dynamics under a laser field, where the exact boundary condition, referred to as an integral boundary condition, is derived with the help of the Green's function for a free electron.  Then, the exact ABCs, which involve a time-convolution, are approximated to avoid the repeated calculations of the integrals. The most predominant method is by finite sums of exponentials in the real-time space \cite{arnold2003,Jiang2004} or rational functions in the Laplace domain or the Fourier space \cite{fevens1999,shibata1991,szeftel2004,xu2006absorbing} to facilitate a fast evaluation. The convergence, stability, and efficiency of these methods have been thoroughly investigated \cite{antoine2008review, han2013}. The most extensively studied case is the one-dimensional case.  The extensions to high dimensional cases  can be found in  \cite{jiang2008,antoine2004,Alpert2002,han2004,Xu2007}; but these methods are often limited only to special geometries of the interior domain. Another classical strategy is the perfectly matched layer (PML) \cite{berenger1994}, where one first constructs a buffer layer so that the outgoing waves in the computational domain are exactly preserved (perfect matching). The most common approach is to introduce a complex stretching of the spatial coordinate to derive a modified equation in the buffer zone, and then the resulting models are discretized simultaneously in the implementation. The PML  has been applied to the nonlinear Schr\"odinger equation by Zheng \cite{zheng2007}. Similar to the CAP methods, the parameters in the PML method have to be calibrated a priori. 



In the three-dimensional case, one challenge in implementing the ABCs stems from the fact that the kinetic energy in the Schr\"{o}dinger equation often needs to be approximated by high-order finite-difference schemes, either to gain sufficient accuracy, or to reduce the computational cost by using larger grid size.  Therefore, the ABCs that are developed from {\it continuous} PDEs, especially those that are expressed in terms of the normal derivatives of the solution at the boundary \cite{bayliss1980radiation,higdon1986absorbing},  are difficult to be integrated with the discretization in the computational domain. For instance, due to the large finite-difference stencil, there are several layers of grids at the boundary, where the ABCs need to be applied, which is non-trivial.  To overcome this difficulty,
 we formulate the ABCs based on a semi-discrete approximation of the TDSE obtained from high-order finite-difference schemes in the entire domain. The resulting boundary condition does not need to be further discretized in space. It can be readily implemented. 
Meanwhile, the time-convolution in the corresponding DtN map is formulated and approximated by rational functions in Laplace domain. This approximation reduces the exact DtN map to ordinary differential equations in the time domain, which is much more efficient in practice. We also tackle the important issue that when the domains is of arbitrary geometry, there is no simple representations of the DtN map, e.g., in terms of a pseudo-differential operator \cite{antonie2004,szeftel2004}. In this case, it is expressed as a matrix-function which involves the Hamiltonian in the exterior. We will employ the discrete boundary-element method  \cite{Li2012} to 
evaluate the DtN map in the rational interpolation. 

We will demonstrate that the first-order ABC obtained from this approach correspond to a CAP method. This offers an interesting connection between the DtN map and an existing  method. But the Hamiltonian that represents the absorbing potential is computed from the exact ABC, rather than empirically built (e.g., a diagonal matrix). We will also show that the second-order ABC is more general than CAP. In most cases, it is more accurate than the first-order approximation. 

This paper is organized as follows: In section \ref{sec:2}, the solution of the exterior problem, which leads to the exact ABC, is formulated. We approximate the exact boundary condition based on its Laplace transform in section \ref{sec:3}. We discuss the choices of the interpolation conditions such that the dynamics is stable. In section \ref{sec:4}, both 1D and 3D numerical simulations are performed to verify the stability and effectiveness. Further, we test the ABCs on a time-dependent Hartree-Fock (TDHF) model with localized interactions. 

\section{Formulation of a discrete DtN map for Schr\"odinger-type of equations}
\label{sec:2}
 This paper is concerned with the time-dependent Schr\"{o}dinger equations in a domain $\Omega \subset \mathbb R^d$, 
\begin{equation}
  \label{eq:1}
  i\frac{\partial }{\partial t}\phi(\bm x, t) = \hat{H}\phi(\bm x, t), \quad \bm x \in \Omega.
\end{equation}
Here we have scaled the Planck constant and mass to unity. $i$ is the imaginary unit. $\phi(\bm x, t)$ is the wave function for an electron or nucleon in the $d$-dimensional quantum system. 
In general, the Hamiltonian operator can be written as $\hat H = -\frac{\nabla^2}2 + V(\bm x,t)$,  with $V$ being the total potential of the system. 
A Schr\"odinger-like equation can be derived from the exact many-body Schr\"odinger equation with various mean-field approximations of the many-body wave function, such as the Hartree-Fock approach \cite{baerends1973self} and the time-dependent density-functional theory \cite{marques2006time}. For simplicity, we illustrate the formulation with the simplest model, and demonstrate the applications of other Schr\"odinger-type equations in section \ref{sec:tdhf}.

We suppose that $\Omega_\i$ is a subdomain of interest in $\Omega$. 
Rather than solving the Schr\"{o}dinger equations in the entire domain $\Omega,$ we 
take into account the influence from the surrounding region, here denoted by $\Omega_\ii$, by deriving ABCs at the boundary of $\Omega_\i$.
$\Omega = \Omega_\i\cup \Omega_\ii$.    In the exterior region $\Omega_\ii$, we assume that the operator $\hat H$ does not explicitly depend on $\bm x$. More specifically, $V(\bm x)  = 0$, or more generally a constant, for any $\bm x\in \Omega_\ii$. 

Most ABCs have been formulated in the continuous case \cite{szeftel2004,antonie2003,baskakov1991implementation,Hellumns1994}, which has to be further discretized in order to be combined with the finite-difference or finite-element approximation in the interior of $\Omega_\i$. For easier implementation,
our approach starts directly with a spatially discrete model. Namely, we first consider the semi-discrete approximation of \eqref{eq:1} in the entire domain, and 
regard it as a full/exact model. This has been done for the acoustic wave equations in \cite{givoli1998discrete}. But efficient evaluations of the ABCs were not fully discussed there.

Let $\{\bm x_j\}$ be the set of $n_\i$ grid points with a constant spacing of $h$ in the interior $\Omega_\i$ and $n_\ii$ grid points in the exterior $\Omega_\ii$. $n_\i\ll n_\ii$. $n_\ii$ can be infinite. In real-space methods, the Hamiltonian operator in \eqref{eq:1} can be approximated  by a representation over numerical grids  (e.g., \cite{octopus}), finite elements (e.g., \cite{motamarri2013higher}), or atom centered basis set. These spatial approximations reduce the problem to a discrete model in space. Here we take the simplest approach, the finite difference method, to demonstrate the formulation of the ABC. To ensure high order accuracy, these finite difference methods often employ large stencils. For example, a one dimensional Laplacian operator, acting on a continuous function $\psi(x,t)$,  can be approximated by,
 \[ \frac{\partial ^2}{\partial x^2} \psi(x,t) \approx  \frac{-\psi_{-2}+16\psi_{-1}-30\psi_{0}+16\psi_1-\psi_2}{12h^2}. \]
Here $\psi_j(t)=\psi(x + jh,t)$ with $h$ being the grid spacing. In the high-dimensional case, this finite-difference formula can be applied in each direction \cite{beck2000real}. 

Let us introduce some notations to make the derivation more transparent. These notations are largely adopted from domain decomposition methods \cite{quarteroni1999domain}. To begin with,  the indices of the nodal points are sorted as $\bm x_1, \dots, \bm x_{n_\i}\in \Omega_\i$, and $\bm x_{n_\i+1}, \dots, \bm x_{n_\i+n_\ii}\in \Omega_\ii$. Due to the discretization in space, the wave functions are written in vectors, and the semi-discrete model can be written in a compact form as follows:
\begin{equation}
  \label{eq:22}
  \left\{
    \begin{aligned}
      i\dot{\phi}_\i(t) &= H_{\ici}\phi_\i(t)+H_{\icii}\phi_\ii(t),\\
      i\dot{\phi}_\ii(t) &= H_{\iici}\phi_\i(t)+H_{\iicii}\phi_\ii(t),
    \end{aligned}
  \right.
\end{equation}
where $\bm \phi_\i = [\phi(\bm x_k)]_{\bm x_k\in \Omega_\i}$ and $\bm \phi_\ii= [\phi(\bm x_k)]_{\bm x_k\in \Omega_\ii}$. $\bm \phi_\i\in \mathbb C^{n_\i}$ and $\bm \phi_\ii\in \mathbb C^{n_\ii}$. $\dot{}$ means the derivative with respect to the time $t$. Following the same partition,  the discretization of the Hamiltonian operator $H$ is structured as follows,
\begin{equation}
  \label{eq:32}
  H =
  \begin{bmatrix}
    H_\ici & H_\icii\\
    H_\iici & H_\iicii
  \end{bmatrix}.
\end{equation} 
Since the Schr\"{o}dinger equation \eqref{eq:1} has been discretized in space, 
we use $\dot{\quad}$ to denote the time derivatives hereafter.  

Our goal is to simplify the second equation in \eqref{eq:22}, while retaining the first equation. 
 $H_\ici$ and $H_\iicii$ are respectively the discretized Hamiltonian operator in $\Omega_\i$ and $\Omega_\ii$. $H_\ici$ can be a nonlinear operator. $H_\icii$ and $H_\iici$ are the off-diagonal blocks representing the coupling between $\Omega_\i$ and $\Omega_\ii$. 
 Notice that due to the finite-difference approximation of the kinetic energy term, all these matrices are sparse.

The boundary condition will be expressed in terms of the values of the wave functions near the boundary of $\Omega_\i$. To this end, we denote $\Gamma$ as the boundary of $\Omega_\i$, such that
\begin{equation}
  \label{eq:25}
  \Gamma=\{ \bm x_j\in \Omega_\i\vert \text{ if there exists } \bm x_k \in \Omega_\ii \text{ such that } H_{jk} \ne 0\}.
\end{equation}

The width of $\Gamma$ depends on the width of the finite-difference stencil.  We define $\bm \phi_\Gamma\in \mathbb C^{n_\Gamma}$ as a vector formed by all $\bm \phi_j$ for $  \bm x_j \in \Gamma$, where $n_\Gamma$ is the number of grid points in $\Gamma$. The vector $\bm \phi_\i$ can be reordered so that the first $n_\Gamma$ components are associated with the grid points in $\Gamma$, i.e.,
\begin{equation}\label{eq: reorder}
\bm \phi_\i =
\begin{bmatrix}
  \bm \phi_\Gamma\\
  \bm \phi_{\i\backslash\Gamma}
\end{bmatrix}.
\end{equation}

The boundary region $\Gamma$ is defined in such a way that there is no direct coupling between the points in the interior of $\Omega_\i$ and those in $\Omega_\ii$. This is reflected in the off-diagonal block of the Hamiltonian, 
\begin{equation}\label{eq: h2g}
H_\iici = \big[H_{\iiciii} \;\; H_{\ii,\i\backslash\Gamma}\big]=\big[H_{\iiciii} \;\; 0\big].
\end{equation}

In general, this procedure can be carried out by defining a restriction operator $E$ to extract the components of a function that correspond to grid points at the boundary $\Gamma$ from a function defined in  $\Omega_\i$. Namely, 
\begin{equation}\label{eq: phig}
    \bm \phi_\Gamma=E\bm \phi_\i.
\end{equation}
With the reordering in \eqref{eq: reorder}, the matrix $E\in \mathbb R^{n_\Gamma\times n_\ii}$ can be explicitly expressed as $E=[I_{n_\Gamma} \;\; 0]$. Further, \eqref{eq: h2g} can now be
simply written as $H_{\iiciii}=H_\iici E^T$.
Since $E E^T = I_{n_\Gamma}$, we also have, 
\begin{equation}\label{eq: hg2}
  H_\icii = E^T H_{\iiicii}.
\end{equation}

\medskip

We are now set to formulate the ABC. We first consider the influence of the wave functions in $\Omega_\ii$ on the wave functions in $\Omega_\i$, and we define 
\begin{equation}
  \bm f_\Gamma=H_{\iiicii}\bm \phi_\ii.
\end{equation}
Then the first equation reads,
\begin{equation}
  \label{eq:omega1}
  \dot{\bm \phi}_\i(t) = -iH_\ici\bm \phi_\i(t) -i E^T \bm f_\Gamma(t).
\end{equation}
Here $ \dot{} $ denotes the derivative with respect to the time variable $t$.

At this point, we observe that what is needed to compute $\bm \phi_\i $ in time  is $\bm f_\Gamma(t).$
For this purpose, let us take the Laplace transform of the second equation in \eqref{eq:22}. Due to the assumption that $\Phi_\ii(0)=0,$ one has,
\begin{equation}
  \label{eq:3}
  is\bm \Phi_\ii(s) = H_\iicii\bm \Phi_\ii(s) + H_\iici\bm \Phi_\i(s),
\end{equation}
where $\bm \Phi_\ii(s) = \mathcal L \{\bm \phi_\ii\}(s)$ and $\bm \Phi_\i(s) = \mathcal L \{\bm \phi_\i\}(s)$. By defining $\widetilde H_\iicii(s) = H_\iicii-isI$, equation \eqref{eq:3} simply reads
\begin{equation}
 \label{eq:sum0}
  \widetilde H_\iicii(s)\bm \Phi_\ii(s) =-H_\iici\bm \Phi_\i(s).
\end{equation}
Formally, its solution can be expressed as
\begin{equation}
  \label{eq:34}
  \bm \Phi_\ii(s) = -\widetilde H_\iicii^{-1}H_\iici\bm \Phi_\i(s).
\end{equation}

 Due to the locality of $H_\iicii$, the right hand side of \eqref{eq:34} can be further reduced. Notice that,
$$H_\iici\bm \Phi_\i=H_{\iiciii} E \bm \Phi_\i = H_{\iiciii}\bm \Phi_\Gamma.$$
The first step is from the identity \eqref{eq: h2g}, and the second step used \eqref{eq: phig}. This simplifies the solution to,
\begin{equation}
  \label{eq:34'}
  \bm \Phi_\ii(s) = -\widetilde H_\iicii^{-1}H_{\iiciii}\bm \Phi_\iii(s).
\end{equation}

Since $H_\iicii$ is a $n_\ii\times n_\ii$ matrix, it is in general impractical to solve the linear system \eqref{eq:34'} numerically. However,
the main observation from \eqref{eq:omega1} is that we only need $\bm f_\Gamma$ to keep the computation in $\Omega_\i$. We let $\bm F_{\Gamma}$
be the Laplace transform of $\bm f_\Gamma$. By left-multiplying the equation above by $H_{\iiicii}$, we obtain,
\begin{equation}
  \label{eq:d2n}
  \bm F_{\Gamma}(s)=K(s)\bm \Phi_{\Gamma}(s).
\end{equation}
The matrix-valued function $K(s): \mathbb{R} \to \Omega_\iii \times \Omega_\iii$, which will play a key role in the numerical approximation,  is given by, 
\begin{equation}
  \label{eq:17}
  K(s) = -H_{\iiicii}\widetilde H_\iicii^{-1}(s)H_{\iiciii}= -H_{\iiicii}\big[ H_\iicii - is I\big]^{-1}H_{\iiciii}.
\end{equation}
 The mapping $K(s)$ is precisely the Dirichlet-to-Neumann (DtN) map of the problem  \eqref{eq:22} in the Laplace domain. Although $K(s)$ still involves 
 the inverse of a large matrix, it is much easier to compute than \eqref{eq:34'} due to the fact that only a small number of entries are needed.
 This is known as selected inversion \cite{linalg}.  The detailed solution to this problem, which involves a discrete boundary element method,  will be discussed in Section \ref{sec:3} and Appendix A. 
 
In the time domain, the DtN map becomes a time convolution,
\begin{equation}
  \label{eq:tbc}
  \bm f_{\Gamma}(t)=\int_0^t \kappa(t-\tau)\bm \phi_{\Gamma}(\tau) d\tau.
\end{equation}

With the DtN map, we incorporate \eqref{eq:tbc} into \eqref{eq:omega1}, and we arrive at a reduced problem \eqref{eq:22}, 
\begin{equation}
  \label{eq:exactBC}
  \dot{\bm \phi}_\i(t) = -iH_\ici\bm \phi_\i(t) -i E^T \int_0^t \kappa(t-\tau)  E \bm \phi_\i(\tau) d\tau,
\end{equation}
where $\kappa(t)$ is the real-time kernel function which corresponds to $K(s)$ in the Laplace domain. 

\medskip

The main difficulties in implementing the transparent boundary condition \eqref{eq:tbc}  are: i) there is no analytical expression for $K(s)$ or $\kappa(t)$ in general, other than the formula \eqref{eq:17} in terms of a large-dimensional matrix, in which case it is expensive to evaluate $K(s)$ or $\kappa(t)$ repeatedly. ii) the direct evaluation of the time-convolution integral adds up quickly to the computational cost: Since the kernel function  $\kappa(t)$ does not have compact support in time, long time integration is required to evaluate the exact boundary condition.

In light of these concerns, we will introduce further approximations in the next section.

\begin{remark}
Our procedure for reducing the problem is reminiscent of reduced-order modeling \cite{Bai2002,benner2015survey}. More specifically, the quantity $\bm f_\Gamma$ can be viewed as
the low-dimensional output, and $\bm \phi_\Gamma$ corresponds to the control quantity. 
 In the reduced-order literature, it is sometimes convenient to write \eqref{eq:3} into the alternative form
\begin{equation}
  \label{eq:19}
  i(s-s_0)(H_\iicii-is_0I)^{-1}\bm \Phi_\ii(s) = \bm \Phi_\ii(s) + (H_\iicii-is_0I)^{-1}H_\iici
  \bm \Phi_\i(s),
\end{equation}
where $s_0\in \mathbb C$ is a pre-selected scalar. In this case, the kernel function becomes
\begin{equation}
  \label{eq:27}
  K(s) = -H_{\iiicii}(I-i(s-s_0)(H_\iicii-is_0I)^{-1})^{-1}(H_\iicii-is_0I)^{-1}H_{\iiciii}.
\end{equation}
If we denote ${\mathcal {A}}=i(H_\iicii-is_0)^{-1}$ and $C=(H_\iicii-is_0I)^{-1}H_{\iiciii}$, the Taylor expansion of the kernel function $K(s)$ around $s=s_0$ reads
\begin{equation}
  \label{eq:28}
  K(s) = -H_{\iiicii}(C + {\mathcal A}C(s-s_0) + {\mathcal A}^2 C(s-s_0)^2 + {\mathcal A}^3C(s-s_0)^3 + \cdots).
\end{equation}
The coefficients (moments) in the expansion are  $M_0=-H_{\iiicii}C$, $M_1=-H_{\iiicii}{\mathcal A}C$, $M_2=-H_{\iiicii}{\mathcal A}^2C$, etc. 
One way to approximate the kernel function $K(s)$ is by rational functions with the same moments, which is known as {\it moment matching} \cite{Bai2002,baker1996pade}. One may also notice that \eqref{eq:28} is connected to a Krylov subspace $\mathcal K({\mathcal A}, C)$. However, for the problem considered here, the higher powers of ${\mathcal A}$ are much more difficult to compute. 
\end{remark}


\section{Approximation of the discrete DtN map}
\label{sec:3}
Consider the following rational functions:
\begin{equation}
  \label{eq:24}
  R_{m,m}(s) = (s^m-s^{m-1}B_0-\cdots-B_{m-1})^{-1}(s^{m-1}A_0+\cdots+A_{m-1}).
\end{equation}
$A_0, \cdots, A_{m-1}$ and $B_0, \cdots, B_{m-1}$ are $n_\Gamma\times n_\Gamma$ matrices to be determined. Here the integer $m\ge 0$ will be referred to as 
the {\it order} of the approximation.

The rational approximation reduces the DtN map in the Laplace domain to
\begin{equation}
  \label{eq:30}
  \bm F_\Gamma(s) = R_{m,m}(s)\bm \Phi_\Gamma(s).
\end{equation}
One advantage of the rational approximation is that in the time domain, the dynamics can be represented by an ODE,
\begin{equation}
  \label{eq:31}
  \bm f_\Gamma^{(m)}=B_0\bm f_\Gamma^{(m-1)}+\cdots +B_{m-1}\bm f_\Gamma+A_0\bm \phi_\Gamma^{(m-1)}+\cdots+A_{m-1}\bm \phi_\Gamma,
\end{equation}
assuming appropriate initial conditions. The superscript $^{(m)}$ denotes the $m$-th derivative. Now, the non-local time-convolution in the DtN map (Eq. \eqref{eq:exactBC}) is replaced by a linear ODE system, which is much more efficient in practical implementations. 
The reduced model (Eq. \eqref{eq:22}) is replaced by,
\begin{equation}
\label{eq:reduced}
  \left\{
    \begin{aligned}
      &\frac{\partial}{\partial t}\bm \phi_\i(t) = -iH_\ici\bm \phi_\i(t) -i H_\icii \bm \phi_\ii(t), \\
      &\bm f_\Gamma^{(m)}=B_0\bm f_\Gamma^{(m-1)}+\cdots +B_{m-1}\bm f_\Gamma+A_0\bm \phi_\Gamma^{(m-1)}+\cdots+A_{m-1}\bm \phi_\Gamma.
    \end{aligned}
  \right.
\end{equation}

The remaining question is how to determine the coefficients $A_i$ and $B_i$. This is done by interpolation. Namely we match $R_{m,m}(s)$ to
$K(s)$ at certain points. In principle we need $2m$ points of $(s_i, K(s_i))$ to determine these coefficients. In general, one may hope that the
accuracy would improve as the number of interpolation points increases. However, a more subtle issue is the stability of the system \eqref{eq:reduced}.  The semi-discrete models are quite similar to molecular dynamics models, for which the stability of ABCs has been analyzed in \cite{li2009stability}. In particular, the Lyapunov functional approach is quite useful.  
In this paper, we will present the guaranteed stability of the zeroth-order $(m=0)$ and first-order approximation $(m=1)$. We did not find
a simple proof of stability for higher order approximations, and we will resort to our numerical simulations to examine the stability property.

\begin{remark}
 We did not pursue a one-point Pad\'{e} approximation, or the Krylov subspace projections, which has been overwhelmingly successful in reduced-order modeling \cite{Bai2002}. The main reason is that in our case, a selected inversion of $(H_\iicii-is_0)$ can be done very efficiently, but the higher order inverse is very difficult. 
\end{remark}

\subsection{Zeroth order approximation.} For the zeroth order approximation, $R_{0,0}$ becomes a constant matrix, here denoted by $M$. We restrict $M$ to be a Hermitian matrix. The corresponding dynamics is
\begin{equation}
  \label{eq:dynamics0}
  \dot{\bm{\phi}}_\i(t)= -iH_\ici\bm{\phi}_\i(t) - iE^TM E \bm \phi_\i(t).
\end{equation}

\begin{theorem}[Stability condition of the zero order approximation]
  The dynamics (Eq. \eqref{eq:dynamics0}) is stable if $M$ is chosen as $K(s_0)$ for an  arbitrary positive value of $s_0$ ($s_0\ge 0$).
\end{theorem}
\begin{proof}
We define a Lyapunov functional as
\begin{equation}
  \label{eq:lyapunov0}
  W(t) = \bm \phi_{\i}^*(t)\bm \phi_\i(t).
\end{equation}

We examine the derivative of the Lyapunov functional, 
\begin{equation}
  \label{eq:8}
  \begin{aligned}
    \frac{d W}{d t} &= \frac{d }{d t}\bm \phi_\i^*\bm \phi_\i + \bm \phi_\i^*\frac{d }{d t}\bm \phi_\i\\
    & = i (\bm \phi_\i^* H_\ici + \bm \phi_\Gamma^*M^*E)\bm \phi_\i -i\bm \phi_\i^*(H_{\ici}\bm \phi_\i+E^TM\bm \phi_\Gamma)\\
    & = i\bm \phi_\Gamma^*(M^*-M)\bm \phi_\Gamma.
  \end{aligned}
\end{equation}
This indicates that if $M$ has a negative definite imaginary part, then $\frac{d W}{d t}\le 0$. To verify this,  we note that
since $-H_\icii[H_\iicii-is_0I]^{-1}H_\iici=-H_\icii(H_\iicii^2+s_0^2I)^{-1}(H_\iicii+is_0I)H_\iici$, $K(s_0)$ always has a negative imaginary part as long as $s_0$ is real and positive. Therefore, the dynamics \eqref{eq:dynamics0} is stable.
\end{proof}

This shows that the stability does not depend on the choice of the interpolation point.

\begin{remark}
The stability analysis also reveals  that the additional term in \eqref{eq:dynamics0} has an imaginary part that acts as a complex absorbing potential (CAP). However, unlike the commonly used CAP methods \cite{riss1996investigation,Riss_1993,MUGA2004357}, this term is derived from the DtN map, and it is dependent of the mesh size, finite difference  formulas, as well as the geometry of the domain. Further, the higher order approximations of the DtN map, which will be presented next,  provide a more general form that goes beyond the CAP methods.
\end{remark}

\subsection{First-order approximation.} For the first-order approximation, we use the rational function $R_{1,1}=(s I-B)^{-1}A$ to approximate the DtN map. The first-order approximated dynamics is
\begin{equation}
  \label{eq:dynamics1}
  \left\{
    \begin{aligned}
      \dot{\bm \phi_\i}(t) &= -iH_\ici\bm \phi_\i(t) - iE^T\bm f_\Gamma(t),\\
      \dot{\bm f}_\Gamma(t) &= B\bm f_\Gamma(t) + A E \bm \phi_\i(t).
    \end{aligned}
  \right.
\end{equation}
To verify the stability, we introduce a Lyapunov functional,
\begin{equation}
  \label{eq:lyapunov1}
  \begin{aligned}
    W(t) = \bm \phi_\i^*(t)\bm \phi_\i(t) + \bm f_\Gamma^*(t)Q\bm f_\Gamma(t),
  \end{aligned}
\end{equation}
where $Q$ is a positive definite matrix to be determined. The derivative of the above Lyapunov functional is
\begin{equation}
  \label{eq:11}
  \begin{aligned}
    \frac{d W}{d t} = & \frac{d}{d t}\bm \phi_\i^*\bm \phi_\i + \bm \phi_\i^*\frac{d}{d t}\bm \phi_\i + \frac{d }{d t}\bm f_\Gamma^*Q\bm f_\Gamma+ \bm f_\Gamma^* Q\frac{d}{d t}\bm f_\Gamma\\
    = &(i\bm \phi_\i^*H_\ici+i\bm f_\Gamma^*E)\bm \phi_\i + \bm \phi_\i^*(-iH_{\ici}\bm \phi_\i-iE^T\bm f_\Gamma) \\
    &+ (\bm f_\Gamma^* B^*+\bm \phi_\i^*E^TA^*)Q\bm f_\Gamma + \bm f_\Gamma^*Q(B\bm f_\Gamma + AE\bm \phi_\i)\\
    = &\bm f_\Gamma^*(iI+QA)E\bm \phi_\i + \bm \phi_\i^*E^T(-iI+A^*Q)\bm f_\Gamma + \bm f_\Gamma^*(B^*Q+QB)\bm f_\Gamma.
  \end{aligned}
\end{equation}
At this point, we observe that if $iI+QA = 0$, {\it and} $B^*Q+QB$ has negative definite real part, then the dynamics \eqref{eq:dynamics1} is stable. The following theorem provides a guideline for the choice of the interpolation points to ensure stability.

\begin{theorem}[Stability condition of the first order approximation]
  The dynamics (Eq. \eqref{eq:dynamics1}) is stable if the coefficients $A$ and $B$ are determined by the interpolations  $\lim_{\lambda \to 0} \frac{d}{d\lambda}R_{1,1}'=\lim_{\lambda \to 0} \frac{d}{d\lambda}K$, where $\lambda=s^{-1}$, and $R_{1,1}(s_1)=K(s_1)$, where $s_1$ is a any positive real number.
\end{theorem}
\begin{proof}
  From $$\lim_{\lambda\to 0}\frac{d}{d\lambda} R_{1,1} = \lim_{\lambda \to 0} \frac{d}{d\lambda}K ,$$ we have $$A = -iH_{\iiicii}H_{\iiciii}.$$
  
  Now we pick $Q=(H_{\iiicii}H_{\iiciii})^{-1}$. $Q$ is symmetric positive definite. The Lyapunov functional (Eq. \eqref{eq:lyapunov1}) is thus positive definite. Let us denote $K_1=K(s_1)$. When $R_{1,1}(s_1)=K(s_1)$, we have, $$B = s_1I - AK_1^{-1}.$$ Hence, $QB=-is_1A^{-1}+iK_1^{-1}$. Next, we show that $QB$ has a negative definite imaginary part. We start with
  \begin{equation}
    \begin{aligned}
      QB+B^*Q &= 2s_1(H_{\iiicii}H_{\iiciii})^{-1}+i(K_1^{-1}-(K_1^*)^{-1}),\\
      &=2s_1(H_{\iiicii}H_{\iiciii})^{-1}+i(K_1^*)^{-1}(K_1^*-K_1)K_1^{-1}.
    \end{aligned}
  \end{equation}
To continue,  we denote $P_1 = (H_\iicii-is_1I)^{-1}H_{\iiciii}$. As a result, $K_1=-H_{\iiciii} P_1$, and
  \begin{equation}
    \begin{aligned}
      K_1^*-K_1 &= -H_{\iiicii}(H_\iicii+is_1I)^{-1}H_{\iiciii}+H_{\iiicii}(H_\iicii-is_1I)^{-1}H_{\iiciii},\\
      & = 2is_1P_1^*P_1.
    \end{aligned}
  \end{equation}
  Therefore we have,
  \begin{equation}
    \label{eq:35}
    \begin{aligned}
      QB + B^*Q &= 2s_1(K_1^*)^{-1}(K_1^*(H_{\iiicii}H_{\iiciii})^{-1}K_1-P_1^*P_1)K_1^{-1}\\
      &=2s_1(K_1^*)^{-1}(P_1^*H_{\iiicii}(H_{\iiicii}H_{\iiciii})^{-1}H_{\iiciii}P_1-P_1^*P_1)K_1^{-1},\\
      &=2s_1(K_1^*)^{-1}P_1^*(H_{\iiciii}(H_{\iiicii}H_{\iiciii})^{-1}H_{\iiicii}-I)P_1 K_1^{-1}.
    \end{aligned}
  \end{equation}
 Since $H_{\iiciii}(H_{\iiicii}H_{\iiciii})^{-1}H_{\iiicii}$ is an orthogonal projection matrix, $H_{\iiciii}(H_{\iiicii}H_{\iiciii})^{-1}H_{\iiicii}-I$ is symmetric negative semi-definite. As a result, $QB+B^*Q$ is negative semi-definite. Therefore, the dynamics (Eq. \eqref{eq:dynamics1}) is stable.
\end{proof}

The interpolation scheme suggested in the theorem involves an interpolation point $s_1>0$ and another point toward $s_2\to +\infty$. Our numerical tests indicate that the approximation is still stable when $s_2$ is finite.

\subsection{Second-order approximation}
The rational function approximation can be extended to higher order. We briefly describe the second-order approximation here, in which case the DtN map is approximated by $R_{2,2}=(s^2I-sB_1-B_0)^{-1}(sA_1+A_0)$, and the dynamics augmented with this ABC is given by,
\begin{equation}
  \label{eq:dynamics2}
  \left\{
    \begin{aligned}
      \dot{\bm \phi}_\i(t) &= -iH_\ici\bm \phi_\i(t) - iE^T\bm f_\Gamma(t),\\
      \ddot{\bm f}_\Gamma(t) &= B_1 \dot {\bm f_\Gamma}(t) + B_0\bm f_\Gamma(t) + A_1 E \dot{\bm \phi}_\i(t) + A_0 E \bm \phi_\i(t).
    \end{aligned}
  \right.
\end{equation}
In the above equation, $\ddot{}$ denotes the second derivative with respect to the time. We have not found a simple proof that provides the stability condition of the dynamics \eqref{eq:dynamics2} in terms of the interpolation points. The  stability will thus be demonstrated  by numerical tests in Section \ref{sec:4}.

\section{Applications \& numerical experiments}
\label{sec:4}
In this section, we test the ABCs with three examples.
For each model, the  details regarding the numerical tests are respectively discussed in section \ref{sec:tdse1d-s}, section \ref{sec:tdse3d} and section \ref{sec:tdhf}. The three problems are briefly summarized as follows:
\begin{itemize}
\item A 1D time-dependent Schr\"odinger equation extensively used as a test example in the literature \cite{antoine2008review}, although our main emphasis is on Schr\"odinger equation in $\mathbb R^3$. 
The system is a 1D free electron,
  \begin{equation}
    \label{eq:tdse1d}
    i\frac{\partial }{\partial t}\psi(x, t) = \wh{H}\psi(x, t), \wh H=-\frac{\partial^2_x}2 \text{ in } \mathbb R,
  \end{equation}
  with the initial condition $\psi^0(x) = \exp(-(x-x_c)^2+ik_0(x-x_c))$.
\item A 3D time-dependent Schr\"odinger equation considered in \cite{antoine2004}:
  \begin{equation}
    \label{eq:tdse3d}
    i\frac{\partial }{\partial t}\psi(\bm x, t) = \wh{H}\psi(\bm x, t),  \wh H=-\frac{\nabla^2}2 \text{ in } \mathbb R^3,
  \end{equation}
  with the initial condition $ \psi^0(x)=\exp(-x_1^2-x_2^2-x_3^2+ik_0x_1)$
\item A 3D time-dependent Hartree-Fock model \cite{flocard1978three}:
  \begin{equation}
    \label{eq:tdhf3d}
    i\frac{\partial }{\partial t}\varphi_j(\bm x, t) = \wh{H}\varphi_j(\bm x, t) \text{ in } \mathbb R^3, \text { for } j=1,\dots, A
  \end{equation}
  with the initial condition $\varphi_j^0$ determined from the ground state. The form of $\wh H$ is given by \eqref{eq:ham_tdhf} later in this section. The 3D TDHF model is a system of nonlinear 3D time-dependent Schr\"odinger equations. The Hamiltonian $\hat H$ depends on the one-particle wave functions.
\end{itemize}

{\bf Integrators.} In general, numerical integrators can be formulated as \cite{castro2004propagators}
\begin{equation}
  \label{eq:16}
  \phi^{(n+1)}=U\phi^{(n)},
\end{equation}
where $U$ is the operator that mimics the time evolution operator. 
For linear problems with time-independent potential,  the exact operator is a matrix exponential,
\begin{equation}
  \label{eq:18}
  U_E(t,t')=\exp(-i\Delta t H(t')).
\end{equation}

One widely used method is the Crank-Nicholson scheme,
\begin{equation}
  \label{eq:20}
  U_{CN} = (1+i\Delta t/2H)^{-1}(1-i\Delta t/2H).
\end{equation}

For the 3D case, it is often impractical to perform the matrix inversion in the Crank-Nicholson scheme. In TDDFT \cite{castro2004propagators,pueyo2018,martin2004},  one classical method is the Taylor expansion of the exact integrator,
\begin{equation}
  \label{eq:21}
  U_{5}= I - iH\Delta t - \frac{1}{2}H^2(\Delta t)^2 + \frac{1}{6}H^3(\Delta t)^3- i\frac{1}{24}H^4 (\Delta t)^4.
\end{equation}
Clearly, the operator $U_5$ is not unitary. However, we will choose $\Delta t$ to be sufficiently small, in which case this integrator is stable and accurate \cite{marques2003}.
This allows us to focus more on the performance of various ABCs.

\subsection{The 1D time-dependent Schr\"odinger equation}
\label{sec:tdse1d-s}
In the first test, we look at a 1D quantum system. The setting of the problem is illustrated in FIG. \ref{fig:1dse_illustration}.
\begin{figure}[h]
  \centering
  \includegraphics[scale=0.4]{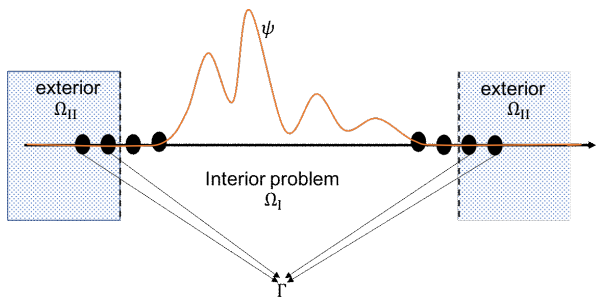}
  \caption{An illustration of the model reduction for one dimensional Schr\"{o}dinger equation. $\psi$ is the wave function, initially supported in  the computational domain $\Omega_\i$.}
  \label{fig:1dse_illustration}
\end{figure}

The analytical solution of Eq. \eqref{eq:tdse1d} can be explicitly written as
\begin{equation}
  \label{eq:14}
  \psi^{\text{ex}}(x, t) = \sqrt{\frac{i}{i-2t}}\exp\left(\frac{-k_0(x-x_c)+k_0^2t-i(x-x_c)^2}{i-2t}\right),
\end{equation}
assuming the initial condition
\begin{equation}
  \label{eq:23}
  \psi^0(x) = \exp(ik_0(x-x_c)-(x-x_c)^2).
\end{equation}
The initial condition $\psi^0$ is localized around $x_c$, which is the center of the wave packet. $k_0$ is the wavenumber. In this test, we set $k_0=5$ and $x_c=-6$. The exact solution, $\psi^{\text{ex}}$, propagates to the right when $k_0>0$. Therefore, we only need to implement an ABC on the right boundary.  A Dirichlet boundary condition will be imposed on the left.

In our simulations, we pick the interior region  to be $\Omega_\i=[-12,3]$ and the exterior domain is $\Omega_\ii=(\infty, -12)\cup(3, \infty)$. The Laplacian operator is discretized by the five-point scheme with grid spacing of $h=0.01$.

The evaluation of the DtN map is discussed in our previous work \cite{Wu2018} using the discrete Green's function. The details of 1D lattice Green's function will be discussed in  the Appendix \ref{sec:gfun1d}.  We select $s=20$ for the zeroth-order approximation, two points $s=10$, and $20$ for the first-order approximation, and four interpolation points, $s=10$, $11$, $20$, $21$ for the second-order approximation. 
The zeroth order approximation corresponds to a complex absorbing potential  over two grid points at the boundary. But we point out that in practice, the latter method can be applied to a much larger buffer region. In accordance with the width of the finite-difference stencil, $4$ and $8$ extra variables are introduced in the first-order and second-order approximations, respectively.

The solution computed with the Dirichlet boundary condition is completely reflected back into the interior region when the wave function propagates to the boundary (FIG. \ref{fig:interp_sol}). 
The zeroth-order approximation causes some reflection when the wave packet first arrives at the boundary, but most of the reflection is eliminated eventually.
The first-order ABC qualitatively captures the transient profile of the exact solution, with some  errors when the wave reaches the boundary. The second-order ABC provides a much more accurate solution. The solution with ABCs is systematically improved as the order of the approximation increases.
\begin{figure}[!htp]
  \centering
  \includegraphics[scale=0.9]{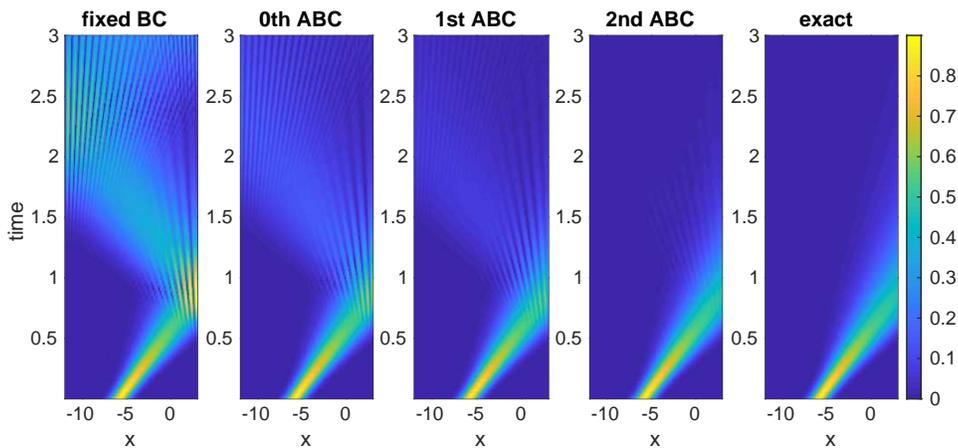}
  \caption{The solutions computed using fixed boundary condition, zeroth-order, first-order and second-order ABCs compared with the exact solution.}
  \label{fig:interp_sol}
\end{figure}
In Fig. \ref{fig:interp_dens}, the total electron number ($L_2$-norm of the wave function in the computational domain) is presented. Before $t=0.5$, the electron number in four cases is the same. The norm of the wave function does not decay with the fixed boundary condition since all the wave function are reflected. The electron number by the first order ABC decays slower than the exact one. In the second order approximation, The maximum error in the $L_2$-norm over time is less than $3\times 10^{-3}$.

In the proposed ABCs, one can either improve the order of ABCs or optimize the interpolation points. In most absorbing techniques, the adjustable parameters are the width of the buffer region or the absorbing strength. It is generally difficult to quantitatively compare the absorbing properties of these methods. As a qualitative comparison, we implemented the following CAP \cite{Riss_1993}
\begin{equation}
  \label{eq:9}
  W(x) =
  \left\{
      \begin{aligned}
        &(x + 16)^2, & -16 < x < -12;\\
        &(x - 7)^2, & 3 < x < 7;\\
        &0, & otherwise
      \end{aligned}
    \right.
\end{equation}
in the 1D model. The size of the buffer region is fixed in our simulations. This leads to the effective Hamiltonian $H^\text{eff}(\eta)=H-i\eta W$ on the interval $[-16, 7]$ where $\eta$ is the CAP strength. The CAP varies from $0.01$ to $10$ (Fig. \ref{fig:interp_dens}). When $\eta$ is smaller than $0.1$, we still observed the reflection from the buffer region. When $\eta=1$, we observed the decent accuracy ($98\%$). As a comparison, this is also achieved by the proposed first-order approximation without optimizing the interpolation conditions, and without introducing a buffer region.  
\begin{figure}[!htp]
  \centering
  \includegraphics[scale=0.8]{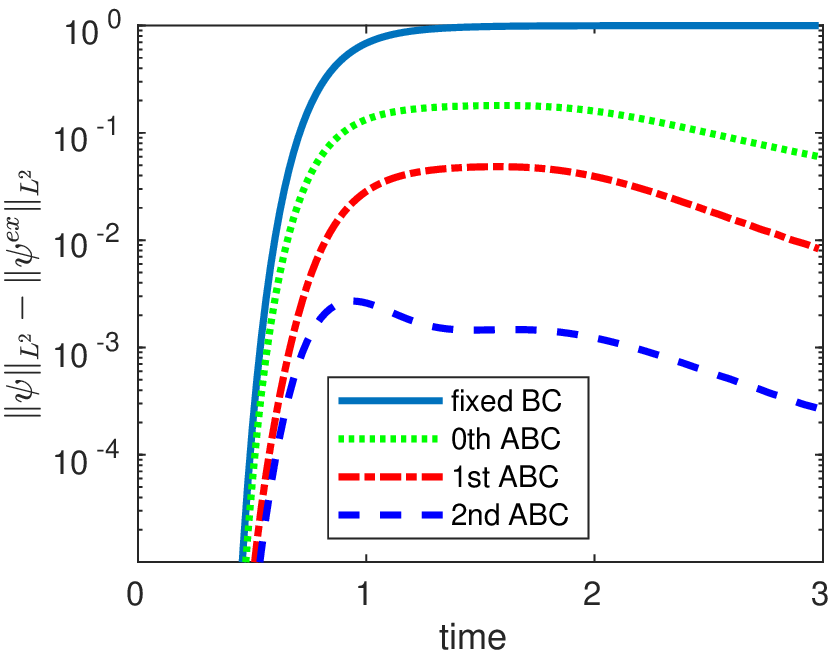}
  \includegraphics[scale=0.8]{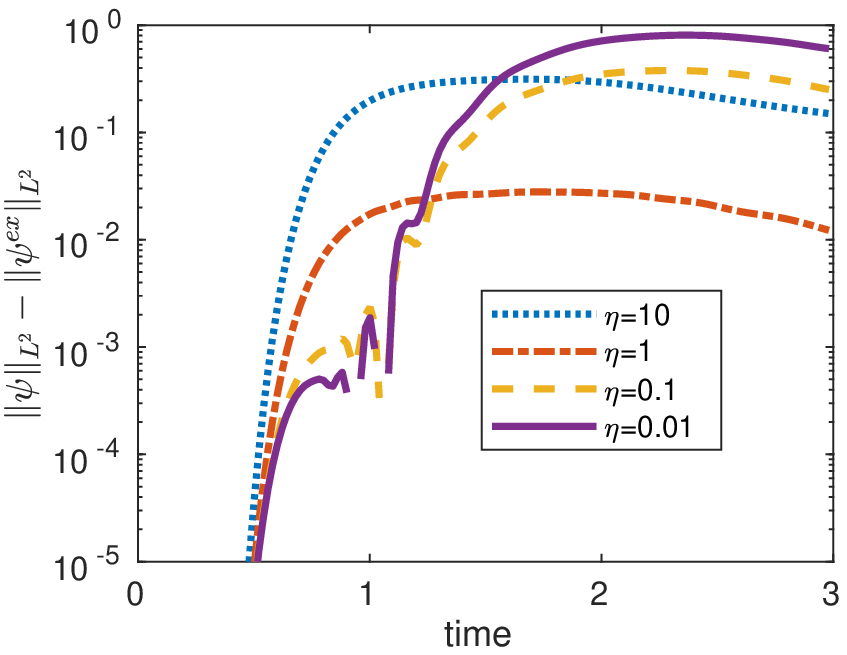}
  \caption{Comparison of the total electron density in $\Omega_\i$ with different ABCs (left). Comparison of the total electron density in $\Omega_\i$ with different complex absorbing potential strength (right)}
  \label{fig:interp_dens}
\end{figure}

\subsection{A 3D time-dependent Schr\"odinger equation}
\label{sec:tdse3d}
In this section, we will test the absorbing boundary conditions for the 3D Schr\"{o}dinger equation for a free electron. We restrict the computational domain in a box $[-1.5, 1.5]\times[-1.5, 1.5]\times[-1.5, 1.5]$. The 3D Laplace operator is approximated by the 7-point finite difference scheme in each direction with uniform grid spacing of $h=0.1$. Consequently, there are $31$ interior points and $6$ exterior points in each axis direction. For each interpolation point, the DtN map $K(t)$ is a $14166\times 14166$ dense matrix. The coefficients of the first-order ABC are computed by interpolation of $s=1,2$. Four interpolation points $s=1,2,3,10$ are used for the second-order ABC.

Similar to the 1D case, we still can construct the analytical solution of Eq. \eqref{eq:tdse3d},
\begin{equation}
  \label{eq:15}
  \psi^{ex} = \left(\frac{i}{i-2t}\right)^{\frac{3}{2}}\exp\left(\frac{-i(x_1^2+x_2^2+x_3^2)- k_0x_1+\frac{1}{2}k_0^2t}{i-2t}\right)
\end{equation}
with $k_0=5$. One should notice that the difference between the analytical solution and the exact solution of the discrete model might not be small due to the large grid spacing. Therefore, we will only use the analytical solution as a reference for qualitative comparisons.

For the time integration, the step size is chosen as $\Delta t = 0.001$. In the first test, we observe how the total electron density in $\Omega_\i$ changes in time (Fig. \ref{fig:se3d}).  The number of electrons is almost a constant when we fix the wave function at the boundary (Dirichlet boundary condition). If we impose the ABC on the system, the number of electrons will decrease after the wave function propagates to the boundary. With the first-order ABC, about $20\%$ of the wave function in magnitude is reflected. Much more reflections are reduced by the second-order ABC. Over a longer time period, such reflections occur multiple times for both the first-order ABC and the second-order ABC. After that, almost all the electrons are emitted out of the box. 
\begin{figure}[!htp]
  \centering
  \includegraphics[scale=0.8]{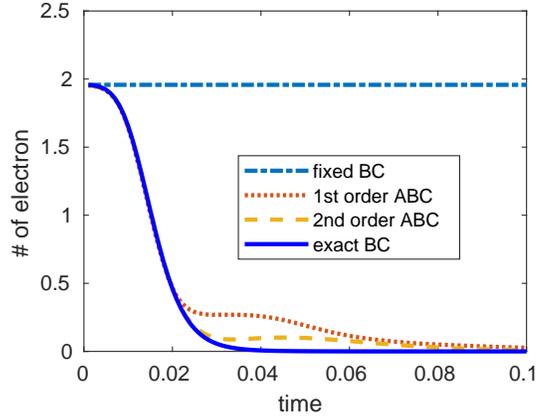}
  \caption{The number of electrons as a function of time.}
  \label{fig:se3d}
\end{figure}

The top view of surface plots of solutions of 3D time-dependent Schr\"odinger equation with different boundary conditions are shown in Fig. \ref{fig:se3d_sol}. The fixed boundary condition leads to significant reflections. From $t=0.011$, it does not provide any significant results. These errors are reduced by the absorbing boundary conditions. Some reflections are still observed in the first-order approximation. The reflections in the second-order ABC are almost negligible.  This experiment also shows that the proposed ABCs is not sensitive to the presence of  corners and edges along the boundary. Even though the wave function propagates to the corners, no significant reflection is observed around the corners. 
\begin{figure}[!htp]
  \centering
  \includegraphics[scale=0.9]{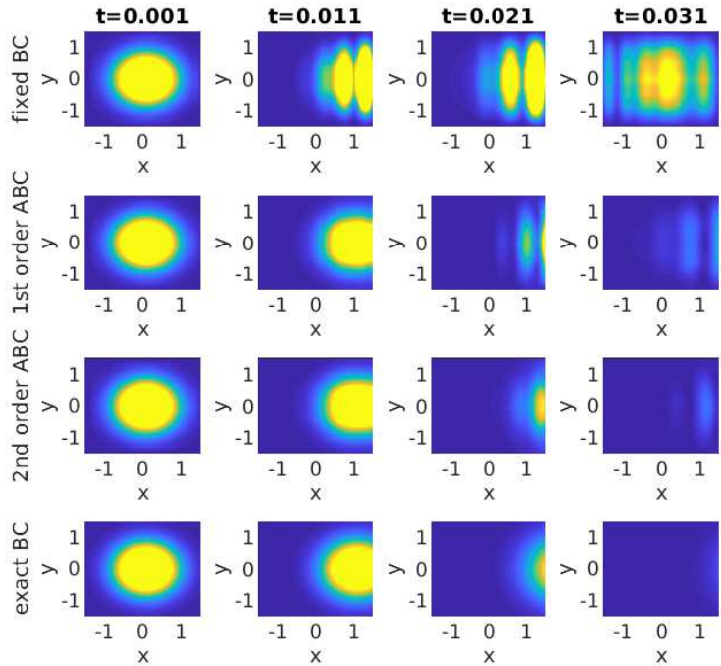}
  \caption[3d TDSE solutions.]{Projection of the 3D electron density on x-y plane. Time-dependent Schr\"odinger equation with the fixed boundary condition (first row), the first order ABC (second row), the second order ABC (third row), and the exact solution (last row). The color indicated the electron density.}
  \label{fig:se3d_sol}
\end{figure}

\subsection{3D time-dependent Hartree-Fock model (TDHF) with localized interactions}
\label{sec:tdhf}
The Hartree-Fock equation for a nucleon system can be formulated from the many-body system by approximating the many-body wave function with the Slater determinant, and applying the variational principle of the Skyrme functional \cite{flocard1978three,Skyrme1977,maruhn1977time} with respect to the wave function. The direct procedure yields the coupled TDHF equations with localized interactions.
\begin{equation}
  \label{eq:tdhf}
  i\hbar \frac{\partial }{\partial t}{\varphi}_j(\bm r, t) = H(t, \rho)\varphi_j(\bm r, t), \quad\quad j=1, \dots, A,
\end{equation}
where $H$ is the time-dependent one-body HF Hamiltonian. The one-body Hamiltonian depends on the nucleon density, given by
\begin{equation}
  \rho(\bm r,t) = \sum_{j=1}^A|\varphi_j(\bm r,t)|^2.
\end{equation} The one-body HF Hamiltonian can be explicitly written as
\begin{equation}
  \label{eq:ham_tdhf}
  H =-\frac{\hbar^2}{2m}\Delta +\frac{3}{4}t_0\rho + \frac{3}{16}t_3\rho^2+W_y+W_C.
\end{equation}
Here $t_0$ and $t_3$ are the coefficients of the Skyrme interactions \cite{flocard1978three,Skyrme1977,maruhn1977time}. 
Among the five terms in the one-body Hamiltonian, the first term is from the kinetic energy. The following two terms are the expectation value of the zero-range density-dependent two-body effective interaction. Furthermore, $W_y$ is the Yukawa potential
\begin{equation}
  W_y(\bm r) = V_0\int d\bm r\frac{\exp{-|\bm r -\bm r'|/a}}{|\bm r -\bm r'|/a} \rho(\bm r'),
\end{equation}
and $W_C$ is the Coulomb potential given by
\begin{equation}
  W_C(\bm r) = e^2\int d\bm r'\frac{1}{|\bm r-\bm r'|}\rho_p(\bm r'),
\end{equation}
where $V_0$ and $a$ are the coefficients of Yukawa interactions. $\rho_p$ is the proton density. In practice, the Yukawa and Coulomb potentials are calculated by solving the corresponding Poisson and Helmholtz problems, respectively,
\begin{equation}
  \left\{
    \begin{aligned}
      \nabla^2 W_c &= -2\pi e^2\rho,\\
      (\nabla^2-1/a^2)W_Y &= -4\pi V_0 a\rho.
    \end{aligned}
  \right.
\end{equation}

We  consider the above TDHF model (\ref{eq:tdhf3d}) in the infinite region $\mathbb R^3$. In practice, the computational domain $\Omega_\i$ (usually a box) is only a small part of the entire $\Omega$. We make the further assumption that
\begin{equation}
  \psi_j = 0  \text{ for } i = 1, \dots, A \text{ and } \rho = 0  \text{ in } \Omega_\ii \text{, at } t = 0.
\end{equation}

Our approach starts from the discrete model. We denote $\bm \varphi_\i$ is the discretization of the wave function $\bm \varphi(\bm x)$ in $\Omega_\i$. $H_\iicii$ is the discretization of the operator $-\frac{\hbar^2}{2m}\nabla^2$. $H_\ici$ is still a nonlinear operator containing the Coulomb potential. The notations of boundaries and DtN map follow those in Section \ref{sec:2}. In this model, we only assume $\rho = 0$ in $\Omega_\ii$ at $t=0$.

As an application, we study the nuclear reaction of the $^{16}\text{O}+^{16}\text{O}$ system in the infinite space $\mathbb R^3$. The setup and data of the numerical experiments are mainly from \cite{flocard1978three}. In this model, we assume the perfect spin-isospin degeneracy for each particle, so that each spatial orbital is occupied by four nucleons. There are $32$ nucleons in total in this system. The two particles at ground state are positioned in a box away from each other and away from the boundary. The ground state is achieved by solving the static Hartree-Fock equation self-consistently in $\Omega_\i$. We assume that there is no interaction between the two particles at the initial state. The Poisson and Helmholtz problems are solved by preconditioned conjugate gradient method using the same discretization method as the wave functions.

The initial condition is given by multiplying each orbital by the phase $e^{i\bm k\cdot \bm r}$ to create a head-on collision. Here, $\bm k$ should be carefully selected to ensure the particles enter into the fusion window. Namely, two particles pass the Coulomb barrier and are trapped by the nuclear force. Otherwise, the two particles move to the boundaries and the nuclear fusion will not occur. We use the same integrator as the case for three dimensional Schr\"{o}dinger equation. The time step is set to be $0.001 fm/c$. In each time step, we also need to perform the self-consistent iteration due to the nonlinearity \cite{pueyo2018,castro2004propagators}. 

In the numerical experiment, $\Omega_\i$ is discretized with grid spacing of $1$ fm. The Laplacian operator is approximated by the 7-point scheme in each spatial direction. The region of interest is $[-15 fm, 15 fm]^3$. We employ the solution of a larger system ($[-30 fm, 30 fm]^3$) as the exact solution to examine the ABCs. The exact number of nucleons and the exact total energy are computed from the larger system restricted to the small region. The initial conditions of the two systems are the nucleon density at the ground state of the smaller system. One should notice that the ground state of the smaller system is not necessarily the ground state of the larger system. The initial condition of the larger system is not the ground state of itself. This will lead a small truncation error due to the long-range interaction in the system. 

The energy conservation and mass conservation (FIG. \ref{fig:tdhf_nucleon_energy}) for the standard TDHF with Dirichlet boundary conditions are easily observed  in our simulations. The system with ABCs released over $1$ MeV from the total energy, and emitted $0.1$ nucleons in the simulating period. More energy and nucleons are expected to be emitted over a longer period. 

Our main focus here is to test how the absorbing property is influenced by  the interpolation points. We present results from the following six cases: (A): $s_1=10^{-1}, s_2=2\times 10^{-1}$; (B): $s_1=10^{0}, s_2=2\times 10^{0}$; (C): $s_1=10^{1}, s_2=2\times 10^{1}$; (D): $s_1=10^{2}, s_2=2\times 10^{2}$; (E): $s_1=10^{-2}, s_2=2\times 10^{-2}$; (F): fixed boundary condition.

The case when $s_1=1$ and $s_2=2$ provides the best agreement with the exact solution in terms of both the number of nucleons and total energy. The case $s_1=0.1$ and $s_2=0.2$ follows second. When the order of magnitudes of $s_1$ and $s_2$ are much larger or smaller than 1, the absorbing properties are much worse. The optimal $s_1$ and $s_2$ should be on the order of 1. From FIG \ref{fig:tdhf_nucleon_energy}, the absorbing property is not sensitive to the selection of $s_1$ and $s_2$, as long as we choose them in the interval $[0.1, 1]$. 
\begin{figure}[!htp]
  \centering
  \includegraphics[scale=0.75]{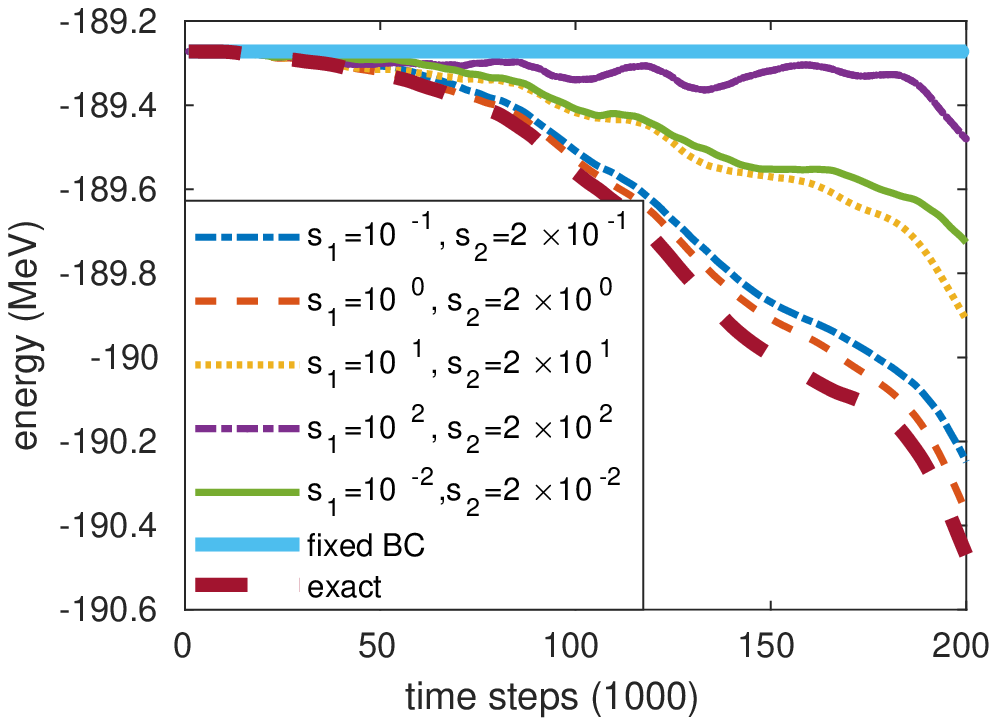}
  \includegraphics[scale=0.75]{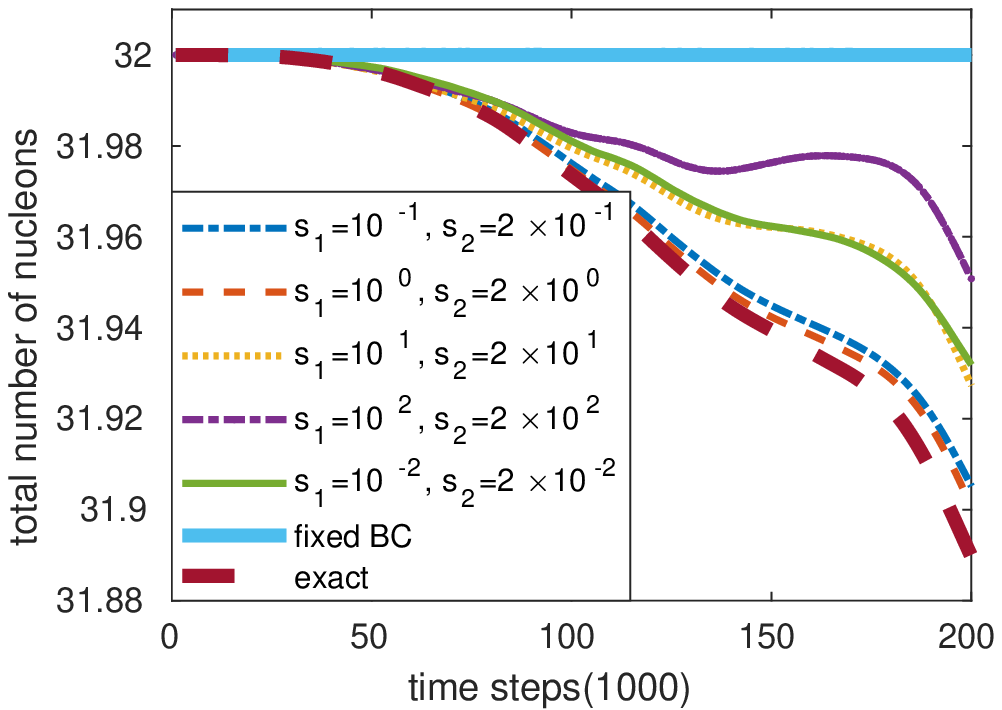}
  \caption{The loss of nucleons and total energy}
  \label{fig:tdhf_nucleon_energy}
\end{figure}

Further, we compare the profile of the densities with the exact solution. Results are displayed in FIG \ref{fig:tdhf_rho}. In the first 5000 time steps, the error of different cases are almost the same. However, when the wave function reaches the boundaries, the Dirichlet boundary condition generates massive errors as expected. The case (A) and (B) still shows great agreement with the exact density. The case (C), (D), and (E) also improved the accuracy of density in varying degrees. 
\begin{figure}[!htp]
  \centering
  \includegraphics[scale=0.9]{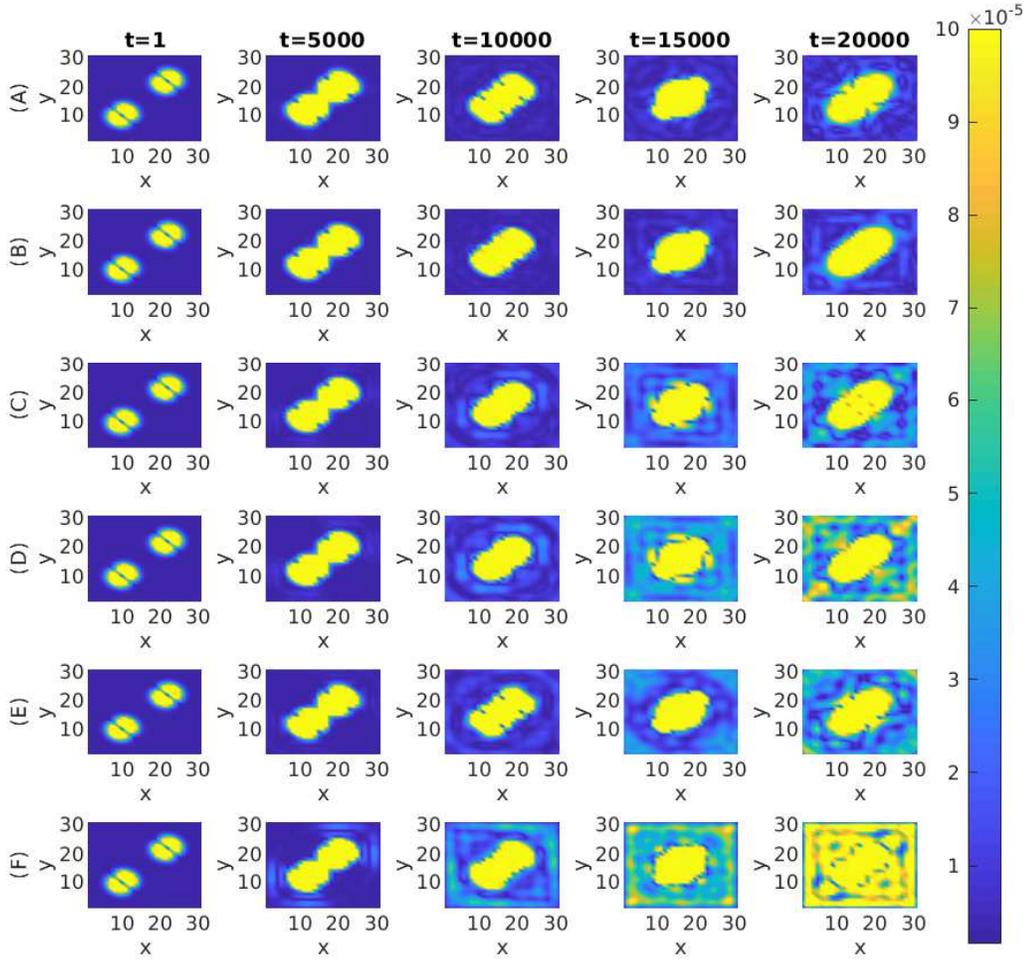}
  \caption{Error of the nucleon density in time evolution of $^{16}\text{O}+^{16}\text{O}$ collision, projected onto the $x-y$ plane. From the top to the bottom, (A): $s_1=10^{-1}, s_2=2\times 10^{-1}$; (B): $s_1=10^{0}, s_2=2\times 10^{0}$; (C): $s_1=10^{1}, s_2=2\times 10^{1}$; (D): $s_1=10^{2}, s_2=2\times 10^{2}$; (E): $s_1=10^{-2}, s_2=2\times 10^{-2}$; (F): fixed boundary condition.}
  \label{fig:tdhf_rho}
\end{figure}

We would like to remark that the long-range potential (Coulomb and Yukawa) make the problem nonlinear in the entire space, in which case the ABCs are difficult to derive, unless further simplifications are made. However, since the two particles never propagate to the boundary, $\rho\approx 0$ outside. We expect that one can neglect the potential terms in the exterior. 

\section{Summary and discussions}
We constructed absorbing boundary conditions for time-dependent Schr\"odinger equations by first deriving the Dirichlet-to-Neumann map. We chose the starting point to be a semi-discrete approximation so that the resulting boundary condition can be readily implemented with the discretization in the interior seamlessly. The nonlocality in time in the exact boundary condition is treated by rational interpolations of the Laplace transform, which in the time domain turns into linear ODEs. For the zeroth and first-order approximations, the stability has been proved. The effectiveness and accuracy are also illustrated by various numerical tests.
For higher-order absorbing boundary conditions, a direct proof seems rather 
challenging. 
Our numerical observations are that second-order and third-order methods are still stable for a wide range of interpolation points $s$. 

In principle, the boundary conditions presented in this paper can be applied to three-dimensional problems with general geometry. The stability results do not depend on the specific configuration of the computational domain. However, a remaining challenge is to choose the  {\it optimal} interpolation conditions  to maximize the accuracy, i.e., minimize the reflection. For one-dimensional problems, or higher dimensional problems in a half-plane, one can compute the reflection coefficients and choose the optimal boundary condition by minimizing the total reflection \cite{li2006variational}. This approach breaks down when corners and edges are present. One possible remedy is   the optimal interpolation strategy from order-reduction problems \cite{anic2008interpolation, gugercin2008}.  For instance, in the current setting, one can formulate the following optimization problem
\begin{equation}
    s_0 =\arg\min  \|K(s)-R_{1,1}(s)\|_{\mathcal H_2},
\end{equation}
for the first-order boundary condition to identify the {\it best} interpolation point. This will be pursued in separate works. 

Another long-standing issue is associated with the long-range interactions among the electrons, i.e., the Coulomb interactions. This issue has been partially addressed in \cite{ehrhardt2006fast}, where the Coulomb potential is replaced by an asymptotic form. But in general, all the existing methods have not be constructed to take full account of the Coulomb potential.

\begin{appendix}
  \section{Evaluation of the Dirichlet-to-Neumann map}
  In this section, we briefly discuss the evaluation of the kernel (Eq. \eqref{eq:17}) in the Laplace domain. The key is to use the Green's function to evaluate the selected inversion. We refer readers to our previous work \cite{Wu2018,Li2012} for the detailed derivations. 
  
  The discrete Green's function $\widetilde G_{ij}$ corresponding to the discretized operator $\widetilde H_{ij}$  satisfies  $$\sum_j\widetilde H_{i,j}\widetilde G_{jk}=\delta_{i,k}, \text{for any integer $i,k$}.$$ To shorten the derivation, the notations of $\widetilde G_{\icii}$ and $\widetilde G_{\iicii}$ follow the convention of $H_{\icii}$ and $H_{\iicii}$ respectively. Furthermore, we introduce the boundary set in $\Omega_\ii$, 
  \begin{equation}
  \Sigma=\{ \bm x_j\in \Omega_\ii\vert \text{ if there exists } \bm x_k \in \Omega_\i \text{ such that } H_{jk} \ne 0\}.
\end{equation}
and matrices $$\widetilde G_{\Gamma, \Sigma}=[\widetilde G_{ij}]_{i\in \Gamma, j\in \Sigma}, \widetilde G_{\Sigma, \Sigma}=[\widetilde G_{ij}]_{i\in \Sigma, j\in \Sigma}, H_{\Gamma, \Sigma}=[H_{ij}]_{i\in \Gamma, j\in \Sigma}, H_{\Sigma, \Gamma}=[H_{ij}]_{i\in \Sigma, j\in \Gamma}.$$
  
We start with the discretized operator $\tilde{H}$ in $\Omega$. By \eqref{eq:3}, the discretized operator satisfies 
  \begin{equation}
    \label{eq:10}
    \widetilde H_\iicii\bm \Phi_\ii + \widetilde H_\iici\bm \Phi_\i=0.
  \end{equation}
  The definition of the discrete Green's function implies that
  \begin{equation}
    \label{eq:4}
    \widetilde G_\iicii\widetilde H_\iicii \bm \Phi_\ii+\widetilde G_\iici\widetilde H_\icii \bm\Phi_\ii = \bm \Phi_\ii.
  \end{equation}
  Eq. \eqref{eq:10} and Eq. \eqref{eq:4} can now be combined into the following equation
  \begin{equation}
    \label{eq:12}
    \bm \Phi_\ii= \widetilde G_\iici  H_\icii \bm\Phi_\ii-\widetilde G_\iicii H_\iici \bm \Phi_\i,
  \end{equation}
  where the fact that $\widetilde H_\icii=H_\icii$ has been used.
   By multiplying both sides of \eqref{eq:12} by $H_{\iiicii}$, we have
  \begin{equation}
    \label{eq:13}
    H_{\iiicii}\bm \Phi_\ii = H_{\iiicii}\widetilde G_{\iicii}H_\icii \bm \Phi_\ii - H_{\iiicii}\widetilde G_{\iicii}H_{\iici}\bm \Phi_\i.
  \end{equation}
  Notice that $\bm f_\Gamma=H_{\iiicii}\bm \Phi_\ii$. Making use of the sparsity of $H_{\iiicii}$ and  $H_\icii$, we obtain 
  \begin{equation}
      \bm f_\Gamma = H_{\Gamma,\Sigma}\widetilde G_{\Sigma,\Gamma} \bm f_\Gamma - H_{\Gamma,\Sigma}\widetilde G_{\Sigma,\Sigma}H_{\Sigma,\Gamma}\bm \Phi_\Gamma
  \end{equation}
  Therefore, the discrete Dirichlet-to-Neumann map is expressed as
  \begin{equation}
    \label{eq:eval}
    \bm f_\Gamma = -(I_{n_\Gamma}-H_{\Gamma,\Sigma}\widetilde G_{\Sigma,\Gamma})^{-1}H_{\Gamma,\Sigma}\widetilde G_{\Sigma,\Sigma}H_{\Sigma,\Gamma}\bm \Phi_\Gamma.
  \end{equation}
 This shows that once the Green's function is available,  the evaluation of the DtN map only involves the operations of small matrices.

  \section{Discrete Green's function for 1D Schr\"{o}dinger equation}
  \label{sec:gfun1d}
  Consider the one-dimensional time-dependent Schr\"{o}dinger equation,
  \begin{equation}
    \label{eq:tdse1d'}
    i\frac{\partial}{\partial t}\psi(t,x) + \frac12 \frac{\partial^2}{\partial x^2}\psi(t,x) = 0. 
  \end{equation}
  By taking the Laplace transform in time ($t\rightarrow s$) and Fourier transform in space ($x\rightarrow q$), the corresponding fundamental solution solves
  \begin{equation}
    \label{eq:2}
   2 siG(s, q) - q^2G(s, q) = 1,
  \end{equation}
  which gives,
  \begin{equation}
    \label{eq:a3}
    G(s, q) = \frac{1}{2si-q^2}.
  \end{equation}
  The inverse Fourier transform yields the fundamental solution in the Laplace domain 
  \begin{equation}
    G(s, x) = -\frac{1}{2\sqrt{-2si}} e^{-\sqrt{-2si}|x|}.
  \end{equation}
  $G(s, x)$ is the continuous Green's function of the following equation in the Laplace domain.
  Now, we turn to the discrete case. To obtain the discrete Green's function, we discretize the Schr\"odinger equation in the Laplace domain by the five-point scheme,
  \begin{equation}
    \label{eq:a4}
   2 si\psi_j(s) + \sum_{k=-2}^2\frac{c_k\psi_{j+k}(s)}{h^2} = 0,
  \end{equation}
  where $c_{-2}=-1/12$, $c_{-1}=4/3$, $c_0=-5/2$, $c_1=4/3$, and $c_2=-1/12$. $h$ is the grid spacing. 

  The characteristic equation of equation \eqref{eq:a4} is given by
  \begin{equation}
    \label{eq:6}
    -\frac{1}{12}u^4+\frac{4}{3}u^3+(i2sh^2-\frac{5}{2})u^2+\frac{4}{3}u-\frac{1}{12}=0.
  \end{equation}
  The four roots of the characteristic equation are given by
  \begin{equation}
    \left\{
      \begin{aligned}
        u_1 &= 4+a-\sqrt{6sh^2i+8a+24},\\
        u_2 &= 4-a+\sqrt{6sh^2i-8a+24},\\
        u_3 &= 4-a-\sqrt{6sh^2i-8a+24},\\
        u_4 &= 4+a+\sqrt{6sh^2i+8a+24},
      \end{aligned}
    \right.
  \end{equation}
  where $a = \sqrt{9+6sh^2i}$, $Re(u_1) < 1$ and $Re(u_3) < 1$ for a small $s$. We take the branch cut ($u_1$ and $u_3$) to ensure the boundedness of the Green's function. Hence, the discrete Green's function can be explicitly expressed as
  \begin{equation}
    G_j(s) = b_1u_1^{|j|} + b_2u_3^{|j|},
  \end{equation}
  where
  \begin{equation}
    \begin{aligned}
      b_1&=\frac{h^2}{2}\sqrt{\frac{3}{(3+2sh^2i)(24+6sh^2i+8a)}},\\
      b_2&=\frac{h^2}{2}\sqrt{\frac{3}{(3+2sh^2i)(24+6sh^2i-8a)}}.
    \end{aligned}
  \end{equation}
  The coefficients $b_1$ and $b_2$ are determined by the conditions $$\sum_{j=-2}^2c_{j}u_j=0 \text{ and } \sum_{j=-2}^{2}c_{j}u_j=1.$$

  We make a comparison between the continuum Green's function and analytical discrete Green's function in FIG. \ref{fig:lgf_1d}. Both the imaginary part and real part of the two Green's functions get close to each other when the distance gets large. This reveals that we can use the continuous Green's function to approximate the discrete Green's function when the distance is large, which becomes quite important especially when there is no explicit formula for the discrete Green's function.
  \begin{figure}[!htbp]
    \centering
    \includegraphics[scale=0.8]{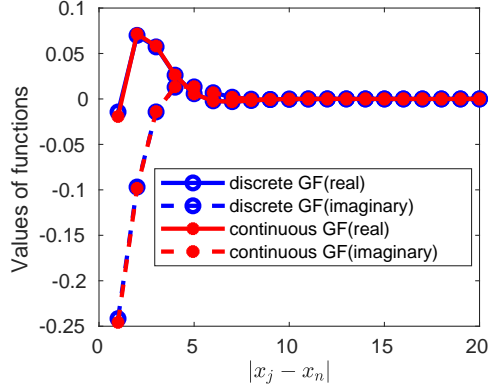}
    \caption{1D discrete Green's function versus 1D continuous Green's function, $s=1, h=1$.}
    \label{fig:lgf_1d}
  \end{figure}
  
  \section{First order ABC for 1D Schr\"{o}dinger equation}
  Let us consider the 1D Schr\"{o}dinger equation (Eq. \eqref{eq:tdse1d'}) on the interval $\Omega_\i=[x_l, x_r]$. We assume the discretization \eqref{eq:a4} of eq. \eqref{eq:tdse1d'}. $x_1,\cdots,x_N$ are $N$ grid points inside the $\Omega_\i$. $x_{-1}$, $x_0$, $x_{N+1}$, and $x_{N+2}$ are the exterior points. The input and output of the DtN map, respectively, are
  \begin{equation}
    \label{eq:29}
    \bm f_\Gamma = 
    \begin{bmatrix}
      f_{-1}\\ f_0\\ f_{N+1}\\ f_{N+2}
    \end{bmatrix}
    \text{ and }
    \bm \Phi_\Gamma = 
    \begin{bmatrix}
      \psi_{-1}\\ \psi_0 \\ \psi_{N+1} \\ \psi_{N+2}
    \end{bmatrix}
  \end{equation}
  The matrices in the DtN map are explictly written as
  \begin{equation*}
    \label{eq:36}
    H_{\Gamma,\Sigma}=
    \begin{bmatrix}
      0 & -1/12 & 0 & 0\\
      -1/12& 4/3 & 0 &0\\
      0 & 0 & 4/3 & 0\\
      0 & 0 & -1/12 & 0
    \end{bmatrix}
    ,
    G_{\Sigma, \Gamma}(s)=
    \begin{bmatrix}
      G_2(s) & G_1(s) & G_N(s) & G_{N+1}(s)\\
      G_3(s) & G_2(s) & G_{N-1}(s) & G_{N}(s)\\
      G_N(s) & G_{N-1}(s) & G_{2}(s) & G_1(s)\\
      G_{N+1}(s)& G_N(s) & G_1(s) & G_2(s)
    \end{bmatrix}
  \end{equation*}
  and
  \begin{equation*}
    \label{eq:38}
    G_{\Sigma, \Sigma}=
    \begin{bmatrix}
      G_0(s) & G_1(s) & G_{N-2}(s) & G_{N-1}(s)\\
      G_1(s) & G_0(s) & G_{N-3}(s) & G_{N-2}(s)\\
      G_{N-2}(s) & G_{N-3}(s) & G_0(s) & G_1(s)\\
      G_{N-1}(s) & G_{N-2}(s) & G_1(s) & G_0(s)
    \end{bmatrix}
  \end{equation*}
  Therefore, the kernel function $K(s)$, which is defined by \eqref{eq:d2n}, can be evaluated \eqref{eq:eval}. Even for the 1D model, the DtN map is not trivial. To demonstrate the idea of the proposed ABCs, we determine the coefficients of the first order ABC is by the moments of $K(s)$. More specifically, the coefficients of the first order ABC are obtained by solving 
  \begin{equation}
    \label{eq:39}
    \left\{
      \begin{aligned}
        (s_1I - B)^{-1}A & = K(s_1)\\
        (s_1I-B)^{-2}A &= -K'(s_1)
      \end{aligned}
    \right.
  \end{equation}
  for certain $s_1>0$. Generally, the coefficients $A$ and $B$ can not be explicitly expressed.
  \section{Discrete Green's function for the 3D Schr\"{o}dinger equation}
  \label{sec:gfun3d}
  Consider the 3D time-dependent Schr\"{o}dinger equation,
  \begin{equation}
    i\partial_t\psi(t,\bm x) + \frac12 \Delta \psi(t,\bm x) = 0.
  \end{equation}
  By following the same procedure, the corresponding fundamental solution in the Laplace space is given by
  \begin{equation}
  G(\bm x, s) = -\frac{1}{4\pi r}e^{-\sqrt{-2si}r},
\end{equation}
where $\bm x = (x_1, x_2, x_3) $ and $r=\sqrt{x_1^2+x_2^2+x_3^2}$.

For the finite-difference approximation, we discretize the operator, e.g., by the  9-point scheme in each spatial direction.
\begin{equation}
  \frac{\partial^2 u}{\partial x^2} \approx \sum_{k=-4}^4\frac{c_ku_{j+k}}{h^2}.
\end{equation}
where $c_{-4}=-1/560$, $c_{-3}=8/315$, $c_{-2}=-1/5$, $c_{-1}=8/5$, $c_0=205/72$, $c_1=8/5$, $c_2=-1/5$, $c_3=8/315$, and $c_4= -1/560$. We denote $c_{\bm k}$ as the coefficients of the 3D discretization. 

The discrete Green's function can be expressed as a Fourier integral, 
\begin{equation}
  G_{nj} = \frac{1}{|B|}\int_{B} \frac{ e^{i(\bm x_n-\bm x_j)\cdot \bm \xi}}{C(\bm \xi)}d\bm \xi, \quad C(\bm \xi) = \sum_{\bm k}c_{\bm k}e^{-i\bm c_{\bm k}\cdot \bm \xi}.
\end{equation}
Here $B$ refers to the Fourier domain associated with the finite-difference grid points. It is given by:  $\frac{\pi}{h}[-1,1]\times\frac{\pi}{h}[-1,1]\times\frac{\pi}{h}[-1,1]$ for uniform grids.

In FIG \ref{fig: g3d}, we show an example of the discrete Green's function,  obtained by a quadrature with $100\times 100 \times 100$ points in the Fourier domain, compared with the continuous Green's function. 
\begin{figure}[!htbp]
  \centering
  \includegraphics[scale=0.8]{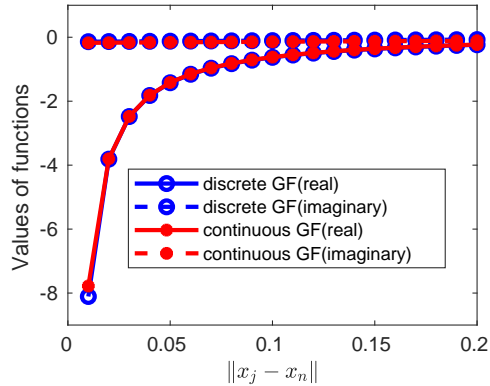}
  \caption{3D discrete Green's function and 3D continuum Green's function, $h=0.01, s=10$.}
    \label{fig: g3d} 
\end{figure}


\end{appendix}

\bibliographystyle{unsrt}
\bibliography{tdqm_update}

\begin{thebibliography}{10}

\bibitem{martin2004}
Richard~M. Martin.
\newblock {\em {Electronic Structure: Basic Theory and Practical Methods}}.
\newblock {Cambridge University Press}, 2011.

\bibitem{helgaker2014}
Trygve Helgaker, Poul Jorgensen, and Jeppe Olsen.
\newblock {\em {Molecular electronic-structure theory}}.
\newblock John Wiley {\&} Sons, 2014.

\bibitem{xie2014complex}
Hang Xie, Yanho Kwok, Feng Jiang, Xiao Zheng, and GuanHua Chen.
\newblock {Complex absorbing potential based Lorentzian fitting scheme and time
  dependent quantum transport}.
\newblock {\em The Journal of chemical physics}, 141(16):164122, 2014.

\bibitem{yabana2006real}
K.~Yabana, T.~Nakatsukasa, J.~I. Iwata, and G.~F. Bertsch.
\newblock {Real-time, real-space implementation of the linear response
  time-dependent density-functional theory}.
\newblock {\em Physica Status Solidi (B) Basic Research}, 243(5):1121--1138,
  apr 2006.

\bibitem{Hellumns1994}
James~R. Hellums and William~R. Frensley.
\newblock Non-{Markovian} open-system boundary conditions for the
  time-dependent {Schr\"odinger} equation.
\newblock {\em Physical Review B}, 49:2904--2906, Jan 1994.

\bibitem{Scrinzi2014}
A.~Scrinzi, H.~P. Stimming, and N.~J. Mauser.
\newblock {On the non-equivalence of perfectly matched layers and exterior
  complex scaling}.
\newblock {\em Journal of Computational Physics}, 269:98--107, {Nov} 2014.

\bibitem{NICOLAIDES197811}
Gy~I. Sz{\'{a}}sz.
\newblock {On the variational calculation of energies and widths of
  resonances}.
\newblock {\em Physics Letters A}, 62(5):313--314, 1977.

\bibitem{SIMON1979211}
Barry Simon.
\newblock {The definition of molecular resonance curves by the method of
  exterior complex scaling}.
\newblock {\em Physics Letters A}, 71(2-3):211--214, 1979.

\bibitem{mccurdy2004solving}
C~W McCurdy, M~Baertschy, and T~N Rescigno.
\newblock {Solving the three-body Coulomb breakup problem using exterior
  complex scaling}.
\newblock {\em Journal of Physics B: Atomic, Molecular and Optical Physics},
  37(17):R137--R187, 2004.

\bibitem{marques2003}
Miguel~AL Marques, Alberto Castro, George~F Bertsch, and Angel Rubio.
\newblock octopus: a first-principles tool for excited electron--ion dynamics.
\newblock {\em Computer Physics Communications}, 151(1):60--78, 2003.

\bibitem{tafipolsky2006}
Maxim Tafipolsky and Rochus Schmid.
\newblock {A general and efficient pseudopotential Fourier filtering scheme for
  real space methods using mask functions}.
\newblock {\em Journal of Chemical Physics}, 124(17):174102, 2006.

\bibitem{Wang201107}
Lin-Wang Wang.
\newblock {Mask-function real-space implementations of nonlocal
  pseudopotentials}.
\newblock {\em Physical Review B - Condensed Matter and Materials Physics},
  64(20):201107(R), nov 2001.

\bibitem{riss1996investigation}
U.~V. Riss and H.~D. Meyer.
\newblock {Investigation on the reflection and transmission properties of
  complex absorbing potentials}.
\newblock {\em Journal of Chemical Physics}, 105(4):1409--1419, 1996.

\bibitem{Riss_1993}
U.~V. Riss and H.~D. Meyer.
\newblock {Calculation of resonance energies and widths using the complex
  absorbing potential method}.
\newblock {\em Journal of Physics B: Atomic, Molecular and Optical Physics},
  26(23):4503--4535, dec 1993.

\bibitem{MUGA2004357}
J.~G. Muga, J.~P. Palao, B.~Navarro, and I.~L. Egusquiza.
\newblock {Complex absorbing potentials}.
\newblock {\em Physics Reports}, 395(6):357--426, 2004.

\bibitem{Soloviev1985}
E~A Soloviev and S~I Vinitsky.
\newblock {Suitable coordinates for the three-body problem in the adiabatic
  representation}.
\newblock {\em Journal of Physics B: Atomic and Molecular Physics},
  18(16):L557----L562, aug 1985.

\bibitem{sidky2000}
E.~Y. Sidky and B.~D. Esry.
\newblock Boundary-free propagation with the time-dependent schr\"odinger
  equation.
\newblock {\em Phys. Rev. Lett.}, 85:5086--5089, Dec 2000.

\bibitem{Riss_1998}
U.~V. Riss and H.~D. Meyer.
\newblock {The transformative complex absorbing potential method: A bridge
  between complex absorbing potentials and smooth exterior scaling}.
\newblock {\em Journal of Physics B: Atomic, Molecular and Optical Physics},
  31(10):2279--2304, may 1998.

\bibitem{He2007}
F.~He, C.~Ruiz, and A.~Becker.
\newblock Absorbing boundaries in numerical solutions of the time-dependent
  {Schr{\"o}dinger} equation on a grid using exterior complex scaling.
\newblock {\em Physical Review A}, 75:053407, May 2007.

\bibitem{Tao2009}
Liang Tao, W.~Vanroose, B.~Reps, T.~N. Rescigno, and C.~W. McCurdy.
\newblock Long-time solution of the time-dependent {Schr{\"o}dinger} equation
  for an atom in an electromagnetic field using complex coordinate contours.
\newblock {\em Physical Review A}, 80:063419, Dec 2009.

\bibitem{Scrinzi2010}
Armin Scrinzi.
\newblock Infinite-range exterior complex scaling as a perfect absorber in
  time-dependent problems.
\newblock {\em Physical Review A}, 81:053845, {May} 2010.

\bibitem{Ge1998}
Jiu-Yuan Ge and John Z.~H. Zhang.
\newblock {Use of negative complex potential as absorbing potential}.
\newblock {\em The Journal of Chemical Physics}, 108(4):1429--1433, 1998.

\bibitem{li2019absorbing}
Xiantao Li.
\newblock Absorbing boundary conditions for time-dependent schr{\"o}dinger
  equations: A density-matrix formulation.
\newblock {\em The Journal of chemical physics}, 150(11):114111, 2019.

\bibitem{antoine2008review}
Xavier Antoine, Anton Arnold, Christophe Besse, Matthias Ehrhardt, and Achim
  Sch{\"a}dle.
\newblock {A Review of artificial boundary conditions for the
  {Schr{\"{o}}dinger} equation}.
\newblock {\em PAMM}, 7(1):1023201--1023202, 2007.

\bibitem{han2013}
Houde Han and Xiaonan Wu.
\newblock {\em {Artificial boundary method}}.
\newblock Springer Science {\&} Business Media, 2013.

\bibitem{han2005finite}
Houde Han and Zhiwen Zhang.
\newblock {An analysis of the finite-difference method for one-dimensional
  Klein-Gordon equation on unbounded domain}.
\newblock {\em Applied Numerical Mathematics}, 59(7):1568--1583, 2009.

\bibitem{antonie2003}
X.~Antoine and C.~Besse.
\newblock {Unconditionally stable discretization schemes of non-reflecting
  boundary conditions for the one-dimensional Schr{\"o}dinger equation}.
\newblock {\em Journal of Computational Physics}, 188(1):157--175, 2003.

\bibitem{alonso2002}
Isaias Alonso-Mallo and Nuria Reguera.
\newblock {Weak Ill-Posedness of Spatial Discretizations of Absorbing Boundary
  Conditions for Schr{\"o}dinger-Type Equations}.
\newblock {\em SIAM Journal on Numerical Analysis}, 40(1):134--158, 2003.

\bibitem{alonso2003}
I.~Alonso-Mallo and N.~Reguera.
\newblock {Adaptive Absorbing Boundary Conditions for Schr{\"{o}}dinger-type
  Equations}.
\newblock In Gary~C Cohen, Patrick Joly, Erkki Heikkola, and Pekka
  Neittaanm{\"{a}}ki, editors, {\em Mathematical and Numerical Aspects of Wave
  Propagation WAVES 2003}, pages 851--856, Berlin, Heidelberg, 2011. Springer
  Berlin Heidelberg.

\bibitem{arnold1998}
A.~Arnold.
\newblock {On absorbing boundary conditions for quantum transport equations}.
\newblock {\em ESAIM: Mathematical Modelling and Numerical Analysis},
  28(7):853--872, 1994.

\bibitem{ermolaev1999}
A.~M. Ermolaev, I.~V. Puzynin, A.~V. Selin, and S.~I. Vinitsky.
\newblock Integral boundary conditions for the time-dependent schr\"odinger
  equation: Atom in a laser field.
\newblock {\em Phys. Rev. A}, 60:4831--4845, Dec 1999.

\bibitem{ermolaev2000}
A.~M. Ermolaev and A.~V. Selin.
\newblock Integral boundary conditions for the time-dependent schr\"odinger
  equation: Superposition of the laser field and a long-range atomic potential.
\newblock {\em Phys. Rev. A}, 62:015401, Jun 2000.

\bibitem{arnold2003}
Anton Arnold, Matthias Ehrhardt, and Ivan Sofronov.
\newblock {Discrete transparent boundary conditions for the Schr{\"o}dinger
  equation: fast calculation, approximation, and stability}.
\newblock {\em Communications in Mathematical Sciences}, 1(3):501--556, 2013.

\bibitem{Jiang2004}
Shidong Jiang and L~Greengard.
\newblock {Fast evaluation of nonreflecting boundary conditions for the
  Schr{\"o}dinger equation in one dimension}.
\newblock {\em Computers {\&} Mathematics with Applications}, 47(6-7):955--966,
  Oct 2004.

\bibitem{fevens1999}
Thomas Fevens and Hong Jiang.
\newblock {Absorbing Boundary Conditions for the Schr{\"{o}}dinger Equation}.
\newblock {\em SIAM Journal on Scientific Computing}, 21(1):255--282, 2003.

\bibitem{shibata1991}
Tsugumichi Shibata.
\newblock {Absorbing boundary conditions for the finite-difference time-domain
  calculation of the one-dimensional Schr{\"{o}}dinger equation}.
\newblock {\em Physical Review B}, 43(8):6760--6763, 1991.

\bibitem{szeftel2004}
J{\'{e}}r{\'{e}}mie Szeftel.
\newblock { Design of Absorbing Boundary Conditions for Schr{\"o}dinger
  Equations in $\mathbb R^d$ }.
\newblock {\em SIAM Journal on Numerical Analysis}, 42(4):1527--1551, 2004.

\bibitem{xu2006absorbing}
Zhenli Xu and Houde Han.
\newblock {Absorbing boundary conditions for nonlinear Schr{\"{o}}dinger
  equations}.
\newblock {\em Physical Review E - Statistical, Nonlinear, and Soft Matter
  Physics}, 74(3):37704, 2006.

\bibitem{jiang2008}
Shidong Jiang and Leslie Greengard.
\newblock {Efficient representation of nonreflecting boundary conditions for
  the time-dependent {Schr{\"{o}}dinger} equation in two dimensions}.
\newblock {\em Communications on Pure and Applied Mathematics}, 61(2):261--268,
  oct 2008.

\bibitem{antoine2004}
Xavier Antoine, Christophe Besse, and Vincent Mouysset.
\newblock {Numerical schemes for the simulation of the two-dimensional
  Schr{\"o}dinger equation using non-reflecting boundary conditions}.
\newblock {\em Mathematics of Computation}, 73(248):1779--1800, Apr 2004.

\bibitem{Alpert2002}
Bradley Alpert, Leslie Greengard, and Thomas Hagstrom.
\newblock {Nonreflecting boundary conditions for the time-dependent wave
  equation}.
\newblock {\em Journal of Computational Physics}, 180(1):270--296, oct 2002.

\bibitem{han2004}
Houde Han and Zhongyi Huang.
\newblock Exact artificial boundary conditions for the {Schr{\"o}dinger}
  equation in {$\mathbb R^2$}.
\newblock {\em Communications in Mathematical Sciences}, 2(1):79--94, 2013.

\bibitem{Xu2007}
Zhenli Xu, Houde Han, and Xiaonan Wu.
\newblock {Adaptive absorbing boundary conditions for Schr{\"o}dinger-type
  equations: Application to nonlinear and multi-dimensional problems}.
\newblock {\em Journal of Computational Physics}, 225(2):1577--1589, {Apr}
  2007.

\bibitem{berenger1994}
Jean~Pierre Berenger.
\newblock {A perfectly matched layer for the absorption of electromagnetic
  waves}.
\newblock {\em Journal of Computational Physics}, 114(2):185--200, {Oct} 1994.

\bibitem{zheng2007}
Chunxiong Zheng.
\newblock {A perfectly matched layer approach to the nonlinear Schr{\"o}dinger
  wave equations}.
\newblock {\em Journal of Computational Physics}, 227(1):537--556, 2007.

\bibitem{bayliss1980radiation}
Alvin Bayliss and Eli Turkel.
\newblock Radiation boundary conditions for wave-like equations.
\newblock {\em Communications on Pure and applied Mathematics}, 33(6):707--725,
  1980.

\bibitem{higdon1986absorbing}
Robert~L Higdon.
\newblock Absorbing boundary conditions for difference approximations to the
  multidimensional wave equation.
\newblock {\em Mathematics of computation}, 47(176):437--459, 1986.

\bibitem{antonie2004}
X~Antoine, C~Besse, and S~Descombes.
\newblock {Artificial boundary conditions for one-dimensional cubic nonlinear
  Schr{\"{o}}dinger equations}.
\newblock {\em SIAM Journal on Numerical Analysis}, 43(6):2272--2293, 2006.

\bibitem{Li2012}
Xiantao Li.
\newblock {An atomistic-based boundary element method for the reduction of
  molecular statics models}.
\newblock {\em Computer Methods in Applied Mechanics and Engineering},
  225-228:1--13, 2012.

\bibitem{baerends1973self}
E.J Baerends, D.E. Ellis, and P.~Ros.
\newblock {Self-consistent molecular Hartree—Fock—Slater calculations {I}.
  The computational procedure}.
\newblock {\em Chemical Physics}, 2(1):41--51, 1973.

\bibitem{marques2006time}
Miguel~AL Marques, Carsten~A Ullrich, Fernando Nogueira, Angel Rubio, Kieron
  Burke, and Eberhard~KU Gross.
\newblock {\em {Time-dependent density functional theory}}, volume 706.
\newblock Springer Science \& Business Media, 2006.

\bibitem{baskakov1991implementation}
VA~Baskakov and AV~Popov.
\newblock Implementation of transparent boundaries for numerical solution of
  the {Schr{\"o}dinger} equation.
\newblock {\em Wave motion}, 14(2):123--128, 1991.

\bibitem{givoli1998discrete}
Dan Givoli, Igor Patlashenko, and Joseph~B. Keller.
\newblock {Discrete Dirichlet-to-Neumann maps for unbounded domains}.
\newblock {\em Computer Methods in Applied Mechanics and Engineering},
  164(1-2):173--185, 1998.

\bibitem{octopus}
Xavier Andrade, David Strubbe, Umberto De~Giovannini, Ask~Hjorth Larsen, Micael
  J.~T. Oliveira, Joseba Alberdi-Rodriguez, Alejandro Varas, Iris Theophilou,
  Nicole Helbig, Matthieu~J. Verstraete, Lorenzo Stella, Fernando Nogueira,
  Alán Aspuru-Guzik, Alberto Castro, Miguel A.~L. Marques, and Angel Rubio.
\newblock {Real-space grids and the Octopus code as tools for the development
  of new simulation approaches for electronic systems}.
\newblock {\em Physical Chemistry Chemical Physics}, 17(47):31371--31396, 2015.

\bibitem{motamarri2013higher}
Phani Motamarri, Michael~R Nowak, Kenneth Leiter, Jaroslaw Knap, and Vikram
  Gavini.
\newblock Higher-order adaptive finite-element methods for {Kohn--Sham} density
  functional theory.
\newblock {\em Journal of Computational Physics}, 253:308--343, 2013.

\bibitem{beck2000real}
Thomas~L Beck.
\newblock Real-space mesh techniques in density-functional theory.
\newblock {\em Reviews of Modern Physics}, 72(4):1041, 2000.

\bibitem{quarteroni1999domain}
A.~Quarteroni and A.~Valli.
\newblock {\em {Domain Decomposition Methods for Partial Differential
  Equations}}.
\newblock Numerical Mathematics and Scie. Clarendon Press, 1999.

\bibitem{linalg}
Lin Lin, Chao Yang, Juan~C. Meza, Jianfeng Lu, Lexing Ying, and Weinan E.
\newblock {SelInv---An Algorithm for Selected Inversion of a Sparse Symmetric
  Matrix}.
\newblock {\em ACM Transactions on Mathematical Software (TOMS)},
  37(4):40:1--40:19, February 2011.

\bibitem{Bai2002}
Zhaojun Bai.
\newblock {Krylov subspace techniques for reduced-order modeling of large-scale
  dynamical systems}.
\newblock {\em Applied Numerical Mathematics}, 43(1-2):9--44, Apr 2002.

\bibitem{benner2015survey}
Peter Benner, Serkan Gugercin, and Karen Willcox.
\newblock A survey of projection-based model reduction methods for parametric
  dynamical systems.
\newblock {\em SIAM review}, 57(4):483--531, 2015.

\bibitem{baker1996pade}
George~A Baker, George~A Baker~Jr, GEORGE~A BAKER~JR, Peter Graves-Morris, and
  Susan~S Baker.
\newblock {\em Pad{\'e} approximants}, volume~59.
\newblock Cambridge University Press, 1996.

\bibitem{li2009stability}
Xiantao Li.
\newblock On the stability of boundary conditions for molecular dynamics.
\newblock {\em Journal of computational and applied mathematics},
  231(2):493--505, 2009.

\bibitem{flocard1978three}
H~Flocard, S~E Koonin, and S~Weiss.
\newblock {Hartree-Fock calculations : Application to $^{16}\mathrm O$ +
  $^{16}\mathrm O$ collisions}.
\newblock {\em Physical Review C}, 17(5):1682, 1978.

\bibitem{castro2004propagators}
Alberto Castro, Miguel~AL Marques, and Angel Rubio.
\newblock Propagators for the time-dependent {Kohn--Sham} equations.
\newblock {\em The Journal of chemical physics}, 121(8):3425--3433, 2004.

\bibitem{pueyo2018}
Adri\'an G\'omez~Pueyo, Miguel A.~L. Marques, Angel Rubio, and Alberto Castro.
\newblock {Propagators for the Time-Dependent Kohn–Sham Equations: Multistep,
  Runge–Kutta, Exponential Runge–Kutta, and Commutator Free Magnus
  Methods}.
\newblock {\em Journal of Chemical Theory and Computation}, 14(6):3040--3052,
  2018.
\newblock PMID: 29672048.

\bibitem{Wu2018}
Xiaojie Wu and Xiantao Li.
\newblock {Stable absorbing boundary conditions for molecular dynamics in
  general domains}.
\newblock {\em Computational Mechanics}, 62(6):1259--1272, 2018.

\bibitem{Skyrme1977}
S.~E. Koonin, K.~T.~R. Davies, V.~Maruhn-Rezwani, H.~Feldmeier, S.~J. Krieger,
  and J.~W. Negele.
\newblock {Time-dependent Hartree-Fock calculations for $^{16}\mathrm{O}$ +
  $^{16}\mathrm{O}$ and $^{40}\mathrm{Ca}$ + $^{40}\mathrm{Ca}$ reactions}.
\newblock {\em Physical Review C}, 15:1359--1374, Apr 1977.

\bibitem{maruhn1977time}
V~Maruhn-Rezwani, KTR Davies, and SE~Koonin.
\newblock {Time-dependent Hartree-Fock calculations for $^{14}\mathrm N$ +
  $^{12}\mathrm C$ reactions}.
\newblock {\em Physics Letters B}, 67(2):134--138, 1977.

\bibitem{li2006variational}
Xiantao Li and Weinan E.
\newblock {Variational boundary conditions for molecular dynamics simulations
  of solids at low temperature}.
\newblock {\em Communications in Computational Physics}, 1(1):135--175, 2006.

\bibitem{anic2008interpolation}
Branimir Anic.
\newblock {\em {An interpolation-based approach to the weighted-$\mathcal{H}_2$
  model reduction problem}}.
\newblock PhD thesis, Virginia Tech, 2008.

\bibitem{gugercin2008}
S.~Gugercin, A.~Antoulas, and C.~Beattie.
\newblock {$\mathcal{H}_2$ Model Reduction for Large-Scale Linear Dynamical
  Systems}.
\newblock {\em SIAM Journal on Matrix Analysis and Applications},
  30(2):609--638, 2008.

\bibitem{ehrhardt2006fast}
Matthias Ehrhardt and Andrea Zisowsky.
\newblock {Fast calculation of energy and mass preserving solutions of
  Schr{\"o}dinger--Poisson systems on unbounded domains}.
\newblock {\em Journal of Computational and Applied Mathematics}, 187(1):1--28,
  2006.

\end{thebibliography}

\end{document}